\documentclass[a4paper, 10pt, leqno]{article}
\usepackage{pkg}
\usepackage{ncmd}
\usepackage{mathtools}
\usepackage{xcolor}
\usepackage{tcolorbox}
\usepackage{comment}

\usepackage[utf8]{inputenc}

\title{LoCCA: Localized Chebyshev Cross Approximation for Kernel Matrix Factorization via Nodal Perturbation Stability\thanks{submitted to the editors \today}}

\author[4,1]{Sumit Singh\,\orcidlink{0009-0002-5581-5349}\thanks{\url{sumit1315singh@gmail.com}, \url{ma22d027@smail.iitm.ac.in}}}
\author[4]{Shrirup Dutta\,\orcidlink{0009-0005-3479-8895} \thanks{\url{shrirupdutta@gmail.com}, \url{ma24d007@smail.iitm.ac.in}}}
\author[1,2,3,4]{Sivaram Ambikasaran\,\orcidlink{0000-0003-2978-6281}\thanks{\url{sivaambi@dsai.iitm.ac.in}, \url{sivaambi@alumni.stanford.edu}}}

\affil[1]{Wadhwani School of Data Science and Artificial Intelligence, IIT Madras, Chennai, India}
\affil[2]{Robert Bosch Center for Data Science and Artificial Intelligence, IIT Madras, Chennai, India}
\affil[3]{Department of Data Science and Artificial Intelligence, IIT Madras, Chennai, India}
\affil[4]{Department of Mathematics, IIT Madras, Chennai, India}

\date{}
\begin{document}

\maketitle

\begin{abstract}
    We propose \textit{Local Chebyshev Cross-Approximation} (LoCCA), a data-driven framework that bridges smooth polynomial interpolation with flexible matrix factorizations. LoCCA dynamically maps ideal grid nodes to their nearest physical neighbors within unstructured point sets. We support this framework with a new perturbation theory, proving that displacing Chebyshev nodes onto physical data points retains mathematical stability and accuracy without explosive error growth. Numerically, LoCCA provides strict, reliable error control even on irregular geometries where standard grid-based methods fail, achieving substantial speedups and near-optimal matrix compression.
\end{abstract}


\paragraph{MSc Classification:} 65F55, 65D05, 41A10, 41A63.

\section{Introduction} \label{sec: Introduction}
Kernel matrices are a fundamental component of modern scientific computing and computational engineering. They arise in a wide variety of computational problems, ranging from the numerical solution of partial differential equations \cite{greengard1987fast, massei2022hierarchical} and integral equations \cite{ho2016hierarchical} to Gaussian process models \cite{ambikasaran2015fast}, inverse problems \cite{ambikasaran2013large}, and kernel-based machine learning methods \cite{cortes1995support}. In many of these applications, the underlying mathematical model is governed by a \textit{kernel function} describing pairwise interactions between points in the computational domain. Let $\mclk:\mclx\times\mcly\rightarrow\bbr$ be a continuous kernel function, where $\mclx,\mcly\subset\bbr^{d}$ are the source and target domains, respectively. Such kernels arise naturally in the formulation of boundary and volume integral equations. A general form of these equations is given by
\begin{equation}\label{equ: fundamental equaion}
    a(x)u(x)+b(x)\int_{\mcly}\mathcal{K}(x,y)c(y)u(y)dy=f(x), \qquad x\in\mclx.
\end{equation}
Here, $a$, $b$, and $c$ are prescribed coefficient functions, $f$ is a given function, and $u$ is the unknown function to be determined. Discretization of \eqref{equ: fundamental equaion} yields a linear system $\mathbf{K}\mathbf{u}=\mathbf{f},$ where $\mathbf{K}\in\mathbb{R}^{N\times N}$ is the dense kernel matrix induced by the kernel function $\mathcal{K}$ and the underlying discrete configurations of physical sampling points $X = \{x_1,\dots,x_N\} \subset \mclx$ and $Y = \{y_1,\dots,y_N\} \subset \mcly$. The matrix $\mathbf{K}$ is given component-wise by $\mathbf{K}_{ij} = \mathbf{K}(x_i, y_j)$,  which can be interpreted as $\mathbf{K} = \mclk(X, Y)$ (in \textsc{Matlab} notation). For simplicity, we restrict our discussion to square kernel matrices, although all subsequent developments extend naturally to the rectangular case.

The kernel matrix $\mathbf{K}$ is generally dense, leading to $\mathcal{O}(N^2)$ storage requirements and computationally expensive matrix operations. Fortunately, while $\mathbf{K}$ is typically full-rank, many of its off-diagonal sub-matrices  $\mathcal{K}(\bar{X}, \bar{Y})$ corresponding to geometrically well-separated or \emph{admissible} subsets $\bar{X} \subset X$ and $\bar{Y} \subset Y$ exhibit numerically low-rank structure \cite{khan2024hodlrdd,singh2025rank}. This observation forms the basis of many modern hierarchical matrix representations and fast kernel algorithms. Consequently, \textit{the efficient computation of low-rank approximations for such sub-matrices is one of the fundamental computational tasks}.

\paragraph{Earlier Works:}
A broad range of techniques has been developed to construct low-rank approximations for the dense kernel matrices. These techniques can be broadly grouped into two classes: \textit{algebraic methods} and \textit{analytic methods}.

Algebraic methods rely solely on the matrix entries and make no assumptions about the analytical properties of the underlying kernel. Classic examples include rank-revealing factorizations, such as the singular value decomposition (SVD), rank-revealing QR (RRQR), interpolative decomposition (ID), and adaptive cross approximation (ACA) \cite{bebendorf2003adaptive}. Among these, ACA has emerged as one of the most widely used techniques for hierarchical matrix construction, constructing the approximation using only selected pivot rows and columns through successive rank-one updates, thereby avoiding explicit access to the full matrix. The convergence depends on the sampled matrix rather than on the kernel's analytical properties, and convergence may be slow or even fail for certain domain configurations \cite{borm2005hybrid}.

In contrast to algebraic methods, analytic methods exploit the kernel's analytical structure to construct low-rank approximations. These methods are based on degenerate kernel expansions of the form $\mathcal{K}(x,y)\approx \sum_{k=1}^{r}\phi_k(x)\psi_k(y)$, from which a low-rank matrix factorization follows immediately upon discretization. Such degenerate approximations can be obtained by Taylor series expansions, multipole expansions \cite{greengard1987fast}, Chebyshev interpolations \cite{fong2009black}, etc. These approaches explicitly construct the basis matrices underlying the low-rank approximation. However, the resulting factorizations do not yield a \textit{skeleton (or $\mathrm{CUR}$) representation}, as the low-rank basis matrices are formed by evaluating the basis functions $\{\phi_k\}$ and $\{\psi_k\}$ rather than being extracted directly from the rows and columns of the kernel matrix.

To overcome this, \textit{hybrid methods} have gained attention in recent years, combining the geometric or analytical robustness of analytic kernels with the data-aware efficiency of algebraic compression. A notable paradigm in this category is the class of \emph{proxy point methods}, introduced to accelerate the Interpolative Decomposition (ID) and Skeletonization processes in fast multipole methods (FMM) and hierarchical matrix factorizations \cite{yesypenko2025simplified, xing2020interpolative, ye2020analytical}. Instead of interactions between a cluster and the full far-field clusters, these approaches evaluate interactions over a small, artificially engineered continuous boundary (e.g., circles or spheres wrapped around a target sub-domain). This analytical proxy representation captures the full row or column span, enabling stable selection of skeleton nodes via algebraic decompositions such as the strong RRQR factorization at minimal cost.

However, Cambier and Darve introduced \emph{Skeletonized Interpolation (SI)} \cite{cambier2019fast}, which explicitly bridges Chebyshev polynomial approximation and matrix skeletons. Rather than using a continuous proxy surface, SI evaluates the target kernel function on rigid Cartesian tensor-product grids of Chebyshev nodes covering the spatial domains and applies strong RRQR factorizations to select optimal skeletons directly from the grid. This approach yields a true skeleton ($\mathrm{CUR}$) representation while inheriting the strong, stable convergence bounds of classical polynomial interpolation theory.

Despite their utility, proxy and skeletonized interpolation frameworks are structurally tethered to idealized geometric layouts. Imposing rigid grids or proxy surfaces onto arbitrary, highly unstructured physical point cloud domains introduces severe geometric friction. This motivates a natural generalization: dynamically displacing nodes of an ideal grid (such as a Chebyshev grid) to their nearest \textit{physical/local} neighbors. However, shifting nodes away from a structured grid violates standard interpolation results, leaving a critical theoretical gap that this work aims to address:

\begin{center}
    \begin{tcolorbox}[
        width=0.91\textwidth,  
        colback=black!5,       
        colframe=black,        
        boxrule=1.2pt,         
        arc=3mm,               
        halign=left,           
        top=1.4mm, bottom=1.2mm, left=2mm, right=2mm, 
        boxsep=1mm             
    ] 
    How can we systematically construct cross-approximations from entirely unstructured point clouds by displacing ideal grid nodes, and can we establish a new, rigorous perturbation theory showing that Chebyshev stability is preserved despite these structural node displacements?
    \end{tcolorbox}
\end{center}

\paragraph{Our Contribution:}

This work introduces two foundational pillars: $(i)$ a constructive polynomial interpolation framework interpreting localized node selections as structured perturbations of ideal Chebyshev grids, and $(ii)$ a data-driven matrix cross-approximation algorithm built directly upon this analytical foundation.

First, we establish a new, comprehensive perturbation theory for Chebyshev polynomial interpolation. Shifting classical Chebyshev grids onto entirely unstructured point configurations typically violates the rigid geometric layouts required by standard constructive approximation proofs. We address this limitation by proving that localized, nearest-neighbor physical node selections can be rigorously analyzed as structured perturbations of ideal Chebyshev nodes, thereby showing that polynomial interpolation stability is preserved despite spatial displacements.

Second, leveraging this theoretical stability, we introduce the \emph{Local Chebyshev Cross-Approximation (LoCCA)} algorithmic framework to bridge abstract function spaces and algebraic low-rank matrix factorizations. LoCCA operates in two distinct phases: $(a)$ the \emph{Localized Chebyshev Nodes Selection (LoC-NS)} protocol filters irregular point clouds into high-quality neighbor sets that inherit the analytical features of a Chebyshev grid, and $(b)$ a localized cross-approximation phase executes Rank Revealing QR (RRQR) factorizations on the resulting submatrix to construct skeleton representation.

\paragraph{Highlights of the Article:}
The distinct theoretical components and computational contributions of this work are summarized as follows:
\begin{itemize}
    \item We prove a stability bound (\autoref{thm: unconditional bound}) for the Lebesgue constant $\tilde{\Lambda}_p$ under a first-kind Chebyshev node grid perturbation $\delta$, establishing that the perturbed operator norm scales polynomial-logarithmically with $p$, satisfying $\tilde{\Lambda}_p \le p^C \Lambda_p = \mathcal{O}(p^C \log p)$.
    
    \item We establish sharp deterministic bounds (\autoref{thm: Stability under node perturbations}) tracking the deviation between the unperturbed and perturbed Lagrange interpolant operators ($\| I_p f - \tilde{I}_p f \|_\linf$), proving it is controlled by the spatial perturbation distance $\delta$ and scales polynomial-logarithmically with $p$ for smooth functions.
    
    \item We extend this perturbation stability rigorously to multivariate continuous hypercubes $\Omega = [-1,1]^d$. 
    
    
    \item We develop a two-stage algorithm that first filters irregular point clouds into high-quality candidate sets via localized neighbor tracking (\autoref{alg: loc_nbds}), and subsequently extracts the low-rank skeleton submatrices utilizing strategic rank-revealing QR factorizations (\autoref{alg: LoCCA}).
    
    \item  We provide a global cross-approximation error (\autoref{thm: LoCCA convergence theorem}) bound on unstructured points in terms of the precision $\varepsilon$, which indicates, if $\varepsilon\to 0$, the approximation converges to the actual kernel $\mclk(x,y)$. 
    

    \item Through extensive numerical benchmarks across different kernels and spatial discretizations, we validate the accuracy, robustness, and computational scalability of the LoCCA framework.
    \begin{itemize}
        \item LoCCA achieves compact error spreads comparable to Skeletonized Interpolation (SI) while eliminating the wide error variance and extreme outlier spikes observed in ACA (\autoref{fig: boxplot_all_kernels_all_tols_TSVD}). Moreover, with similar accuracy, LoCCA achieves up to a $10\times$ speedup over ACA (\autoref{fig: time vs grid}).
        \item 
        Crucially, we observe failures of classical Skeletonized Interpolation (SI) in unstructured settings (\autoref{fig: arc_dot_domains_data} and \autoref{fig: concentric_domains_data}), where rigid tensor-product constraints induce sample starvation or severe geometric misalignment, thereby validating the necessity of our localized perturbation approach.
    \end{itemize}
    
\end{itemize}

\paragraph{Outline of the Article:}
The remainder of the paper is organized as follows. \autoref{sec: Preliminaries} reviews the fundamental background on interpolation properties. \autoref{sec: The Algorithm} presents the algorithmic formulation of LoC-NS and LoCCA algorithms. \autoref{sec: Theoretical Analysis} provides the rigorous mathematical proofs for our node perturbation bounds, multivariate extensions, and convergence properties. \autoref{sec: Numerical Experiments} validates our theoretical findings across diverse scientific kernels and domains, followed by concluding remarks in \autoref{sec: conclusion}.

\section{Preliminaries} \label{sec: Preliminaries}

In this section, we introduce the mathematical foundations, notation, and structural metrics necessary for our analytical and algorithmic developments. These elements serve as the formal prerequisites for studying interpolation operators under node perturbations. Frequently used notations are mentioned in \autoref{tab: notations}.

\begin{table}[ht]
\centering
\renewcommand{\arraystretch}{1.12}
\begin{tabular}{ll}
\hline
\textbf{Notation} & \textbf{Description} \\
\hline

$\mclx,\mcly$  & Source and Target Domains in $\bbr^d$ \\

$\mclk$ & Continuous kernel function defined on $\mclx \times \mcly$  \\

$X,Y$ & Discrete point clouds sampled from $\mclx$ and $\mcly$ \\

$n$ & Number of discretization points per spatial direction \\

$N$ & Number of Grid points in each domain (i.e., $N = n^d$)\\

$\mathbf{K}$ & Kernel matrix with entries $\mathbf{K}_{ij}=\mathcal{K}(x_i,y_j)$ \\

$p$ & Number of Chebyshev nodes per spatial direction \\

$\mathcal{C}_{\mclx}, \mathcal{C}_{\mcly}$ & Tensor-product grids of Chebyshev nodes over $\mclx$ and $\mcly$ \\

$P$ & Size of Tensor-product grids of Chebyshev nodes, i.e., $P = \#(\mathcal{C}_{\mclx}) = \#(\mathcal{C}_{\mcly}) = p^d $ \\

$\xi_i^{(\mclx)}, \xi_j^{(\mcly)}$ & Individual Chebyshev nodes in $\mclx$ and $\mcly$\\


$\mathcal{N}_\mclx,\mathcal{N}_\mcly$ & Collections of local neighbors selected around Chebyshev nodes from $X$ and $Y$ \\

$\Lambda_p, \tilde{\Lambda}_p$ & Unperturbed and perturbed univariate Lebesgue constants, respectively \\

\hline
\end{tabular}
\caption{Frequently used notations.}
\label{tab: notations}
\end{table}

\subsection{Basic Definitions}


\begin{mydfn}[Fill Distance]
Let $\Omega \subset \mathbb{R}^d$ be a bounded domain. For a discrete point cloud configuration $X = \{\mathbf{x}_i\}_{i=1}^N \subset \Omega$, the \emph{fill distance} (or mesh size) $h_{X}$ measures the maximum geometric coverage gap over the domain and is defined as $h_{X} = \sup_{\mathbf{y} \in \Omega} \min_{\mathbf{x}_i \in X} \|\mathbf{y} - \mathbf{x}_i\|_2.$
\end{mydfn}

\begin{mydfn}[Local Neighborhood of a Node Set]
Given a dense physical point cloud $X = \{\mathbf{x}_i\}_{i=1}^N \subset \Omega$ and a discrete target node set $\mathcal{C} = \{\boldsymbol{\xi}_k\}_{k=1}^P \subset \Omega$, the \emph{local neighborhood of size $l$} around a specific target node $\boldsymbol{\xi}_k \in \mathcal{C}$ within $X$ is denoted by $\mathcal{N}_l(\boldsymbol{\xi}_k)$ and defined as a subset satisfying $\mathcal{N}_l(\boldsymbol{\xi}_k) \subset X$ and $\#(\mathcal{N}_l(\boldsymbol{\xi}_k)) = l$, such that $  \max_{\mathbf{x} \in \mathcal{N}_l(\boldsymbol{\xi}_k)} \|\boldsymbol{\xi}_k - \mathbf{x}\|_2 \le \min_{\mathbf{x}' \in X \setminus \mathcal{N}_l(\boldsymbol{\xi}_k)} \|\boldsymbol{\xi}_k - \mathbf{x}'\|_2.$

The global collection of unique physical neighbors associated with the entire target node set $\mathcal{C}$ is given by the coordinate union $\mathcal{N}_\mathcal{C} = \bigcup_{k=1}^P \mathcal{N}_l(\boldsymbol{\xi}_k).$
\end{mydfn}

\begin{mydfn}[Supreme and $C^1$ Norms]
For a continuous scalar-valued function $f: \Omega \to \mathbb{R}$ belonging to the Banach space $C(\Omega)$, we employ the continuous supremum (uniform) norm defined as $\|f\|_{L^\infty(\Omega)} = \sup_{\mathbf{x} \in \Omega} |f(\mathbf{x})|$, which we abbreviate as $\|f\|_\linf$ when the domain is clear from context. 

Furthermore, we define the standard norm on the first-order continuously differentiable space $C^1(\Omega)$ as
\begin{equation}
    \|f\|_{C^1(\Omega)} = \|f\|_{L^\infty(\Omega)} + \sum_{k=1}^d \left\| \frac{\partial f}{\partial x_k} \right\|_{L^\infty(\Omega)}.
\end{equation}
\end{mydfn}
\begin{mydfn}[Numerical Rank]
Let $\mathbf{A} \in \mathbb{R}^{m \times n}$ be a matrix, and let its singular values be ordered as $\sigma_1 \ge \sigma_2 \ge \dots \ge \sigma_{\min\{m,n\}} \ge 0$. Given tolerance $\varepsilon > 0$, the \emph{numerical rank} $r_\varepsilon = \mathtt{rank}(\mathbf{A}, \varepsilon)$ of the matrix is defined as $r_\varepsilon = \max \left\{ k : \sigma_{k} \ge \varepsilon \sigma_1 \right\}.$
\end{mydfn}

\subsection{Chebyshev Interpolation and Node Perturbations}
Let $\{\xi_i\}_{i=1}^{p}$ denote the Chebyshev nodes of the first kind on $[-1,1]$, defined as the roots of the $p$-th degree Chebyshev polynomial $T_p(x)$
\begin{equation}\label{equ: cheb nodes}
    \xi_i = \cos \theta_i, \quad \text{where } \theta_i = {(2i-1)\pi}/{2p},\;  \text{ for all } i = 1, \dots, p.
\end{equation}
Let $m:= \min_{k\neq l} |\xi_k - \xi_l|$ denote the minimum separation, then it is straightforward to obtain the following bounds on $m$,
\begin{equation}\label{equ: grid spacing}
    \pi^2/p^2 \ge m \ge {\pi^2}/{2p^2}.
\end{equation}

Consider the Lagrange interpolation operator associated with Chebyshev nodes ${I}_p: C[-1,1] \to \mathbb{P}_{p-1}$, and induced by the $\linf$ norm, the operator norm $\|I_p\|_{\infty}$ represents the worst-case amplification of the functional values, which is quantified by the Lebesgue constant $\Lambda_p$ \cite[Theorem 1]{gunttner1980evaluation}

\begin{equation}\label{equ: lebesgue constant bound}
    \|{I}_p\|_\infty = {\Lambda}_p = \max_{x \in [-1,1]} \sum_{i=1}^p |l_i(x)| = \mclo{\log p} ,
\end{equation}
where the fundamental perturbed Lagrange basis polynomials on Chebyshev nodes are defined by ${l}_i(x) = \prod_{j \neq i} \frac{x - {\xi}_j}{{\xi}_i - {\xi}_j}$.

Let $\{\tilde{\xi}_i\}_{i=1}^{p} \subset [-1,1]$ be a set of perturbed nodes satisfying the proximity condition $|\tilde{\xi}_i - \xi_i| < \delta$ for all $i = 1, \dots, p$. To preserve node order and maintain a strict separation, we enforce the threshold constraint
\begin{equation}\label{equ: delta strict threshold}
    \delta \le {m}/{4}.
\end{equation}

\begin{myremark}\label{rmk: threshold_constant_generalization}
    The choice of the constant $\delta \leq m/4$ in \eqref{equ: delta strict threshold} is made without loss of generality and for presentation clarity, as it yields the convenient lower bound $|\tilde{\xi}_i - \tilde{\xi}_j| \ge \frac{1}{2}|\xi_i - \xi_j|$ in \eqref{equ: separation delta lower}. The theoretical results hold identically for any $\delta \le c \cdot m$ with a fixed constant $c \in (0, 1/2)$. 
\end{myremark}

The Lagrange interpolation operator associated with this perturbed node configuration is denoted by $\tilde{I}_p: C[-1,1] \to \mathbb{P}_{p-1}$, and its operator norm is given exactly by its respective Lebesgue constant
\begin{equation}\label{equ: perturbed lebesgue constnt definition} 
    \|\tilde{I}_p\|_\infty = \tilde{\Lambda}_p = \max_{x \in [-1,1]} \sum_{i=1}^p |\tilde{l}_i(x)|,
\end{equation}
where the fundamental perturbed Lagrange basis polynomials are defined by
\begin{equation}
    \tilde{l}_i(x) = \prod_{j \neq i} \frac{x - \tilde{\xi}_j}{\tilde{\xi}_i - \tilde{\xi}_j}.
\end{equation}

We next establish two foundational lemmas regarding harmonic summations over Chebyshev configurations, which are crucial for the theoretical analysis presented in \autoref{sec: Theoretical Analysis}.

\begin{mylma}\label{lma: chebyshev-harmonic}
Let $\{\xi_j\}_{j=1}^p$ be the first-kind Chebyshev nodes. Then there exists a positive constant $C_H > 0$, independent of $p$ and the target node index $k$, such that
\begin{equation}
    \sum^p_{\substack{j = 1, j \neq k}} \frac{1}{|\xi_j - \xi_k|} \le C_H p^2 \log p, \qquad \text{for each } k = 1, \dots, p.
\end{equation}
\end{mylma}

\begin{proof}
Fix $k \in \{1, \dots, p\}$. The spatial distance for $j \neq k$ can be expressed as
\begin{equation} \label{equ: | xi j - xi k|}
    |\xi_j - \xi_k| = 2 \left| \sin\left(\frac{\theta_j + \theta_k}{2}\right) \right| \left| \sin\left(\frac{\theta_j - \theta_k}{2}\right) \right|.
\end{equation}
Since $\theta_j, \theta_k \in (0, \pi)$, simple algebraic manipulation yields
\begin{equation} \label{equ: sin (theta j +- theta k) bound}
    \left| \sin\left(\frac{\theta_j + \theta_k}{2}\right) \right| \ge \sin\left(\frac{\pi}{2p}\right) \ge \frac{1}{p}, \;\text{ and }\;  \left| \sin\left(\frac{\theta_j - \theta_k}{2}\right) \right| = \sin\left(\frac{|j-k|\pi}{2p}\right) \ge \frac{|j-k|}{p}.
\end{equation}
Now, substituting the lower bounds \eqref{equ: sin (theta j +- theta k) bound} into \eqref{equ: | xi j - xi k|}, we get the Chebyshev node seperation bound as
\begin{equation}\label{equ: cheb node seperaion lower bound}
    |\xi_j - \xi_k| \ge 2 \cdot \frac{1}{p} \cdot \frac{|j-k|}{p} = \frac{2|j-k|}{p^2}.
\end{equation}
Taking the reciprocal and summing over all non-coincident indices $j \neq k$ yields the following
\begin{equation}
    \sum^p_{\substack{j = 1, j \neq k}} \frac{1}{|\xi_j - \xi_k|} \le \frac{p^2}{2} \sum_{j=1, j \neq k}^p \frac{1}{|j-k|} \le p^2 \sum_{r=1}^{p-1} \frac{1}{r}.
\end{equation}
Now the asymptotic growth of the harmonic numbers $H_{p-1} = \sum_{r=1}^{p-1} r^{-1} = \mathcal{O}(\log p)$ completes the lemma.
\end{proof}

\begin{mylma}\label{lma: restricted-sum}
Let $x \in [-1,1]$ and let $\xi_k$ be closest to $x$. Then there exists a uniform constant $C_1 > 0$, independent of $p$ and the index $k$, such that the grid-receptive sum satisfies
\begin{equation}
    \sum_{\substack{j = 1, j \neq k}}^p \frac{1}{|x - \xi_j|} \le C_1 p^2 \log p.
\end{equation}
\end{mylma}

\begin{proof}
By definition, since $\xi_k$ represents the closest Chebyshev node to the evaluation point $x$, we have $|x - \xi_k| \le |x - \xi_j|$ for all valid indices $j$. Applying the triangle inequality yields
\begin{equation}
    |\xi_j - \xi_k| \le |x - \xi_j| + |x - \xi_k| \le 2|x - \xi_j|,
\end{equation}
which implies $\frac{1}{|x - \xi_j|} \le \frac{2}{|\xi_j - \xi_k|}$. Now summing over $j \neq k$ and using 
\autoref{lma: chebyshev-harmonic} we have
\begin{equation}
    \sum_{\substack{j = 1 , j \neq k}}^p \frac{1}{|x - \xi_j|} \le 2 \sum_{\substack{j = 1, j \neq k}}^p \frac{1}{|\xi_j - \xi_k|} \le 2 C_H p^2 \log p.
\end{equation}
Setting $ C_1: = 2 C_H$ completes the proof.
\end{proof}

\begin{myremark}\label{rmk: threshold-log-bound}
    In the context of perturbation of Chebyshev nodes, the threshold condition $\delta \le m/4$, the scaling behavior typically satisfies a grid-spacing constraint of the form $\delta = \mathcal{O}(p^{-2})$ (by \eqref{equ: cheb node seperaion lower bound}). Under such an assumption, there exists a uniform constant $C_2 > 0$ such that $\delta p^2 \le C_2$. Consequently, \autoref{lma: restricted-sum} directly yields the dampened logarithmic bound:
    \begin{equation}
        \delta \sum_{\substack{j = 1 , j \neq k}}^p \frac{1}{|x - \xi_j|} \le C_1 (\delta p^2) \log p \le C_0 \log p,
    \end{equation}
    where $C_0 := C_1 C_2 > 0$ is a uniform constant independent of $p$.
\end{myremark}

\section{The Algorithm} \label{sec: The Algorithm}

In this section, we introduce numerical algorithms based on perturbed Chebyshev nodes. Rather than forcing a rigid grid onto unstructured datasets, our formulation treats selected local points directly as perturbed configurations of an ideal Chebyshev node. This translation bridges continuous function-space stability and discrete matrix operations through a clean, two-stage framework, the \emph{Localized Chebyshev Nodes Selection} (LoC-NS) method 
and the \emph{Localized Chebyshev Cross Approximation} (LoCCA). 

\subsection{Localized Chebyshev Nodes Selection (LoC-NS)}
First, the LoC-NS method searches the unstructured data sets to find the real physical points closest to these continuous reference nodes, as formalized in \autoref{alg: loc_nbds}. A representative geometric layout of the selected nodes is illustrated in \autoref{fig: domain configuration details}.

\begin{algorithm}[ht]
\caption{Localized Chebyshev Nodes Selection (LoC-NS)}
\label{alg: loc_nbds}
\begin{algorithmic}[1]
\State \textbf{Input:} Grids $X \subset \mclx$, $Y \subset \mcly$, Chebyshev grid size $P$, number of neighbors $l$
\State \textbf{Output:} Selected unique neighbor sets $\mathcal{N}_\mclx \subset X$ and $\mathcal{N}_\mcly \subset Y$ 

\State Construct Chebyshev nodes $\{\xi_i^{(\mclx)}\}_{i=1}^{P} \subset \mclx$ and $\{\xi_j^{(\mcly)}\}_{j=1}^{P} \subset \mcly$

\For{$i=1$ to $P$}
    \State Find $l$ nearest neighbors of $\xi_i^{(\mclx)}$ in $X$
    \State Find $l$ nearest neighbors of $\xi_i^{(\mcly)}$ in $Y$
\EndFor

\State Collect all selected neighbors and retain the unique ones to construct:
\[ \mathcal{N}_\mclx \subset X \quad \text{and} \quad \mathcal{N}_\mcly \subset Y \]

\State \Return $\mathcal{N}_\mclx, \mathcal{N}_\mcly$
\end{algorithmic}
\end{algorithm}

\begin{figure}[ht]
    \centering
    \includegraphics[width=0.7\linewidth]{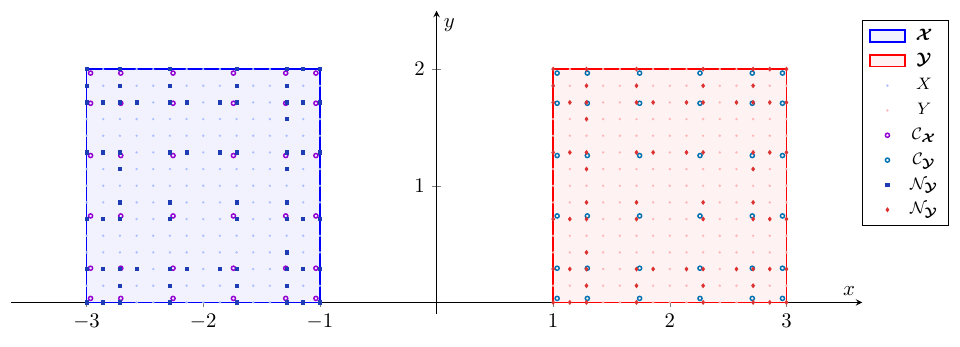}
    \caption{Geometric illustration of the LoC-NS procedure. The domains $\mclx=[-3,-1]\times[0,2]$ (blue) and $\mcly=[1,3]\times[0,2]$ (red) are discretized into the point sets $X$ and $Y$, respectively. Applying \autoref{alg: loc_nbds} maps the Chebyshev nodes to nearby physical points, yielding the filtered point sets $\mathcal{N}_{\mclx}$ and $\mathcal{N}_{\mcly}$. The example is shown for $N=15$, $P=6^2$, and $l=2$.}
    \label{fig: domain configuration details}
\end{figure}

\autoref{alg: loc_nbds} returns the filtered point sets $\mathcal{N}_{\mclx}$ and $\mathcal{N}_{\mcly}$, satisfying $n_x = \#(\mathcal{N}_{\mclx}) \le l P$ and $n_y = \#(\mathcal{N}_{\mcly}) \le lP$.  These reduced point sets preserve the geometric distribution of the reference Chebyshev nodes while remaining entirely within the original unstructured data, making them well-suited for the subsequent matrix skeletonization step.

\subsection{Localized Chebyshev Cross Approximation (LoCCA)}
Using the filtered sets $\mathcal{N}_{\mclx}$ and $\mathcal{N}_{\mcly}$, we build the matrix $\mathbf{K} = \mathcal{K}(\mathcal{N}_{\mclx}, \mathcal{N}_{\mcly}) \in \mathbb{R}^{n_x \times n_y}$. 
We apply the Rank-Revealing QR (RRQR) factorization to the low-rank skeletons, as formalized in \autoref{alg: LoCCA}.

\begin{algorithm}[ht]
\caption{Localized Chebyshev Cross Approximation (LoCCA)}
\label{alg: LoCCA}
\begin{algorithmic}[1]
\State \textbf{Input:} Grids $X \subset \mclx$, $Y \subset \mcly$, kernel $\mathcal{K}$, $P$,  $l$, tolerance $\varepsilon$
\State \textbf{Output:} Low-rank skeleton index sets $\widetilde{\mathcal{N}}_\mclx$ and $\widetilde{\mathcal{N}}_\mcly$

\State $[\mathcal{N}_\mclx, \mathcal{N}_\mcly] \leftarrow \text{\textbf{Call \autoref{alg: loc_nbds}}}(X, Y, P, l)$ 

\State Construct the cross matrix: 
\[ \mathbf{K} = \mathcal{K}(\mathcal{N}_\mclx, \mathcal{N}_\mcly) \]

\State Perform RRQR on $\mathbf{K}$ with tolerance $\varepsilon$ to get: $\widetilde{\mathcal{N}}_\mcly \subset \mathcal{N}_\mcly$
\vskip 0.15cm
\State Perform RRQR on $\mathbf{K}^\top$ with tolerance $\varepsilon$ to get: $\widetilde{\mathcal{N}}_\mclx \subset \mathcal{N}_\mclx$

\State \Return $$\mathcal{K}(X,Y) \approx \mathcal{K}(X,\widetilde{\mathcal{N}}_\mcly) \; \mathcal{K}(\widetilde{\mathcal{N}}_\mclx,\widetilde{\mathcal{N}}_\mcly)^{-1} \; \mathcal{K}(\widetilde{\mathcal{N}}_\mclx,Y)$$
\end{algorithmic}
\end{algorithm}

The numerical rank of the cross-approximation is given by $r = \max\{\#(\widetilde{\mathcal{N}}_{\mclx}), \#(\widetilde{\mathcal{N}}_{\mcly})\}$. Although in most of the experiments we noticed $\#(\widetilde{\mathcal{N}}_{\mclx}) = \#(\widetilde{\mathcal{N}}_{\mcly})$, there might be chances of size mismatches. To resolve this, the smaller index set is expanded to reduce the error in the approximation.


To establish the robustness of the LoCCA algorithmic framework, along with the theoretical results, we systematically conducted several numerical experiments over a wide range of kernel functions and different domains. We primarily focus on standard kernels of the form
\begin{equation}\label{equ: kernel list}
    \mclk(x,y) = \kappa(\|x - y\|_2), \; x \in \mclx, \; x \in \mcly, \quad \text{where, } \kappa(r) \in \left\{ \frac{1}{r}, \, \log r, \, \frac{\cos r}{r}, \, e^{-r^2}, \, \sqrt{1+r^2} \right\}.
\end{equation}
    

\section{Theoretical Analysis}\label{sec: Theoretical Analysis}

In this section, we present the mathematical foundation establishing the uniform stability, dimensional scaling, and convergence properties of the algorithm. We also analyze how the localized coordinate shifts executed in \autoref{alg: loc_nbds} propagate the low-rank approximation.

\subsection{Stability and Error Bounds under Perturbations}

We begin this section with one-dimensional results on polynomial interpolation on perturbed Chebyshev nodes, which we interpret as the mapping of ideal Chebyshev nodes to their nearby discrete physical neighbors. We now state the following result on uniform stability of the perturbed Lebesgue constant (defined in \eqref{equ: perturbed lebesgue constnt definition}).

\begin{theorem}\label{thm: unconditional bound}
Let the unperturbed first-kind Chebyshev nodes be given by \eqref{equ: cheb nodes} and let the perturbed nodes satisfy $|\tilde{\xi}_i - \xi_i| < \delta$ subject to the uniform threshold constraint \eqref{equ: delta strict threshold}. Then there exists a positive constant $C > 0$, independent of $p$, such that the perturbed Lebesgue constant satisfies
\begin{equation}
    \tilde{\Lambda}_p \le p^C \Lambda_p = \mathcal{O}(p^C \log p) \qquad \text{as } p \to \infty.
\end{equation}
\end{theorem}

\begin{proof}
Express the perturbed nodes as $\tilde{\xi}_i = \xi_i + \epsilon_i$ with $|\epsilon_i| < \delta \le m/4$. Fix $x \in [-1,1]$ and let $\xi_k$ be the Chebyshev node nearest to $x$. By the triangle inequality and the separation condition \eqref{equ: delta strict threshold}, we have
\begin{equation}\label{equ: separation delta lower}
    |\tilde{\xi}_i - \tilde{\xi}_j| \ge |\xi_i - \xi_j| - 2\delta \ge \frac{1}{2}|\xi_i - \xi_j|, \qquad \text{for } i \neq j.
\end{equation}
We decompose each fundamental perturbed Lagrange basis polynomial $\tilde{l}_i(x)$ as follows
\begin{equation}\label{equ: basis rational decomposition}
    \tilde{l}_i(x) = l_i(x) \cdot \prod_{j \neq i} \frac{x - \tilde{\xi}_j}{x - \xi_j} \cdot \prod_{j \neq i} \frac{\xi_i - \xi_j}{\tilde{\xi}_i - \tilde{\xi}_j}, \quad \text{ provided } x\neq \xi_k.
\end{equation}
Using the inequality $\log(1+t) \le t$ for $t \ge 0$, combined with \eqref{equ: separation delta lower} and \autoref{lma: chebyshev-harmonic}, we get
\begin{equation}
    \log \prod_{j \neq i} \frac{|\xi_i - \xi_j|}{|\tilde{\xi}_i - \tilde{\xi}_j|} \le \sum_{j \neq i} \frac{|\epsilon_i - \epsilon_j|}{|\tilde\xi_i - \tilde\xi_j|} \le \sum_{j \neq i} \frac{4\delta}{|\xi_i - \xi_j|} \le 4 C_H \delta p^2 \log p.
\end{equation}
Since $\delta p^2 \le C_2$, we have the denominator quotient of $\tilde{l}_i(x)$ strictly bounded by a polynomial factor
\begin{equation}\label{equ: den quotient bound}
    \prod_{j \neq i} \frac{\xi_i - \xi_j}{\tilde{\xi}_i - \tilde{\xi}_j} \le e^{4 C_H C_2 \log p} = p^{C_1}.
\end{equation}

Now we evaluate the numerator-tracking component of $\tilde{l}_i(x)$. The following three cases we need to show, then we are done. 

\begin{description}
    \item[Case I: ] $ |x - \xi_k|\ge \delta $. \\ For each index $j \neq i, k$, we write $|x - \tilde{\xi}_j| \le |x - \xi_j| + \delta = |x - \xi_j| \left(1 + \frac{\delta}{|x - \xi_j|}\right)$.
    Isolating the localized nearest-node factor $k$, we take the product for $j \neq i$ as 
    \begin{gather}
        \prod_{j \neq i} |x - \tilde{\xi}_j| \le \left( \prod_{j \neq i} |x - \xi_j| \right) \left( 1 + \frac{\delta}{|x - \xi_k|} \right) \prod_{j \neq i, k} \left( 1 + \frac{\delta}{|x - \xi_j|} \right) 
    \end{gather}
    Using $1+t \le e^t$ alongside the bound established in \autoref{rmk: threshold-log-bound}, yeilds
    \begin{equation}\label{equ: 1 + delta by (x - xi j) bound}
        \prod_{j \neq i,k} \frac{|x - \tilde{\xi}_j|}{|x - \xi_j|}\leq \prod_{j \neq i, k} \left(1 + \frac{\delta}{|x - \xi_j|}\right) \le \exp\left( \delta \sum_{j \neq i, k} \frac{1}{|x - \xi_j|} \right) \le e^{C_0 \log p} = p^{C_0}.
    \end{equation}
    For the index $j = k$, the condition $|x - \xi_k|\ge \delta $ ensures that $\left( 1 + \frac{\delta}{|x - \xi_k|} \right)\leq 2$. Using the above bounds, from \eqref{equ: basis rational decomposition} we get
    \begin{equation}
        |\tilde{l}_i(x)| \le 2 \cdot p^{C_1} \cdot p^{C_0} \cdot |l_i(x)| = 2 p^C |l_i(x)|.
    \end{equation}
    Summing over all components $i = 1, \dots, p$ and maximizing over the evaluation domain $x \in [-1,1]$ gives
    \begin{equation}
        \tilde{\Lambda}_p \le C_3 p^C \Lambda_p.
    \end{equation}
    Now, by \eqref{equ: lebesgue constant bound}, as  $\Lambda_p = \mclo{\log p}$ ensures that $\tilde{\Lambda}_p = \mathcal{O}(p^C \log p)$.

    \item[Case II:] $0 < |x - \xi_k| < \delta $. \\
    We evaluate the components relative to the nearest unperturbed node $\xi_k$.

    \textit{Sub-case A: $i = k$}. In this case, the term $j = k$ is naturally omitted from the product in \eqref{equ: basis rational decomposition}. The remaining non-local terms are handled similarly to Case I, yielding
    \begin{equation}
        |\tilde{l}_k(x)| \le p^{C_1 + C_0} |l_k(x)|.
    \end{equation}

    \textit{Sub-case B: $i \neq k$}. In this case, rearranging the product \eqref{equ: basis rational decomposition} and using \eqref{equ: den quotient bound} and \eqref{equ: 1 + delta by (x - xi j) bound}, we get
    \begin{equation}\label{equ: case 2 rational direct isolating}
        |\tilde{l}_i(x)| \le p^{C_1 + C_0} \left| l_i(x) \cdot \frac{x - \tilde{\xi}_k}{x - \xi_k} \right|.
    \end{equation}
    Applying the Mean Value Theorem directly to $l_i(x)$ on the localized interval between $x$ and $\xi_k$, we obtain $l_i(x) = l_i'(\eta_i)(x - \xi_k),$ where $\eta_i$ lies strictly between $x$ and $\xi_k$. Substituting this into \eqref{equ: case 2 rational direct isolating} 
    and using Markov's polynomial inequality on the standard baseline derivative, yielding $\|l_i'\|_\linf \le p^2 \|l_i\|_\linf \le p^2 \Lambda_p$ \cite[Chapter 4, \S 1)]{devore1993constructive}. This simplifies the basis bound to $|\tilde{l}_i(x)| \le p^{C_1 + C_0} \left( 2 \delta p^2 \Lambda_p \right)$.
    
    Now summing over all components $i = 1,\dotsc,p$, and maximizing over $x \in [-1,1]$ yields
    \begin{equation}
        \tilde{\Lambda}_p
        \le p^{C_1 + C_0} \left( \Lambda_p + 2 \delta p^3 \Lambda_p \right).
    \end{equation}
    Utilizing the condition $\delta p^2 \le C_2$, the right-hand side of the above simplifies to $p^C \Lambda_p$, establishing a clean polynomial-logarithmic growth of $\mathcal{O}(p^C \log p)$.

    \item[Case III:] $|x - \xi_k| = 0$, i.e., $x = \xi_k$. \\
    At $x = \xi_k$, the perturbed Lagrange basis polynomial is $\tilde{l}_i(\xi_k) = \prod_{j \neq i} \frac{\xi_k - \tilde{\xi}_j}{\tilde{\xi}_i - \tilde{\xi}_j}$. To exploit our established bounds, we rewrite the basis polynomial as follows
    \begin{equation}\label{equ: case 3 factored structure tight}
        |\tilde{l}_i(\xi_k)| =  \prod_{j \neq i} \frac{|\xi_i - \xi_j|}{|\tilde{\xi}_i - \tilde{\xi}_j|}  \times \prod_{j \neq i} \frac{|\xi_k - \tilde{\xi}_j|}{|\xi_i - \xi_j|}.
    \end{equation}
    The first product is uniformly bounded as given by \eqref{equ: den quotient bound}, and to analyze the second product, we consider the following two cases for the index $i$.

    \textit{Sub-case A: $i = k$}. In this case, the second product can be rewritten as 
    \begin{equation}
        \prod_{j \neq k} \frac{|\xi_k - \tilde{\xi}_j|}{|\xi_k - \xi_j|} \le \prod_{j \neq k} \frac{|\xi_k - \xi_j| + \delta}{|\xi_k - \xi_j|} = \prod_{j \neq k} \left(1 + \frac{\delta}{|\xi_k - \xi_j|}\right).
    \end{equation}
    Applying $1+t \le e^t$ along with \autoref{lma: chebyshev-harmonic}, we get
    \begin{equation} \label{equ | 1 + dela by mod ( xi k - xi j) bound}
        \prod_{j \neq k} \left(1 + \frac{\delta}{|\xi_k - \xi_j|}\right) \le \exp\left(\delta \sum_{j \neq k} \frac{1}{|\xi_k - \xi_j|}\right) \le e^{C_H \delta p^2 \log p} \le p^{C_H C_2} = p^{C_4}.
    \end{equation}
    Combining these components back into \eqref{equ: case 3 factored structure tight}, we have $|\tilde{l}_k(\xi_k)| \le p^{C_1 + C_4}$.

    \textit{Sub-case B: $i \neq k$}. In this case, the index $j = k$ appears inside the product, and isolating it, we get
    \begin{equation}
        \prod_{j \neq i} \frac{| \xi_k - \tilde{\xi}_j |}{| \xi_i - \xi_j |} =  \frac{| \xi_k - \tilde{\xi}_k |}{| \xi_i - \xi_k |} \times \prod_{j \neq i, k} \frac{| \xi_k - \tilde{\xi}_j |}{| \xi_i - \xi_j |}.
    \end{equation}
    Given $| \xi_k - \tilde{\xi}_k | = | \epsilon_k | < \delta$, the isolated term reduces to $\frac{\delta}{| \xi_i - \xi_k |}$, and the product can be rewritten as
    \begin{equation} \label{equ: | xi k - tilde xi j | product}
        \prod_{j \neq i, k} \frac{| \xi_k - \tilde{\xi}_j |}{| \xi_i - \xi_j |} =  \prod_{j \neq i, k} \frac{| \xi_k - \xi_j |}{| \xi_i - \xi_j |} \times \prod_{j \neq i, k} \frac{| \xi_k - \tilde{\xi}_j |}{| \xi_k - \xi_j |}.
    \end{equation}
    If $\omega_{\boldsymbol{\xi}}(x) = \prod_{m=1}^p (x - \xi_m)$, then from \autoref{rmk: cheb derivatives}, we have
    \begin{equation}
        \prod_{j \neq i, k} \frac{| \xi_k - \xi_j |}{| \xi_i - \xi_j |} = \frac{|\omega_{\boldsymbol{\xi}}'(\xi_k)| / |\xi_k - \xi_i|}{|\omega_{\boldsymbol{\xi}}'(\xi_i)| / |\xi_i - \xi_k|} = \frac{|\omega_{\boldsymbol{\xi}}'(\xi_k)|}{|\omega_{\boldsymbol{\xi}}'(\xi_i)|} = \frac{\sqrt{1-\xi_i^2}}{\sqrt{1-\xi_k^2}} \leq p.
    \end{equation}

    By \eqref{equ | 1 + dela by mod ( xi k - xi j) bound}, the second product term of \eqref{equ: | xi k - tilde xi j | product} is bounded by $p^{C_4}$. 
 
    Now, for $i \neq k$, from \eqref{equ: case 3 factored structure tight} we get
    \begin{equation}
        |\tilde{l}_i(\xi_k)| \le p^{C_1} \cdot \frac{\delta}{| \xi_i - \xi_k |} \cdot p \cdot p^{C_4} = p^{C_1 + C_4 + 1} \frac{\delta}{| \xi_i - \xi_k |}.
    \end{equation}
    Summing over $i = 1 : p $ with $i \neq k$ and using \autoref{lma: chebyshev-harmonic} we get
    \begin{equation}
        \sum_{i \neq k} |\tilde{l}_i(\xi_k)| \le p^{C_1 + C_4 + 1} \left( \delta \sum_{i \neq k} \frac{1}{| \xi_i - \xi_k |} \right) \le C_H C_2 p^{C_1 + C_4 + 1} \log p.
    \end{equation}
    Combining the results for both Sub-cases A and B proves that $\sum_{i=1}^p |\tilde{l}_i(\xi_k)| = \mathcal{O}(p^C \log p)$. 
\end{description}

This completes the proof.
\end{proof}

\begin{myremark}\label{rmk: sharp_lebesgue_conjecture}
    The bound $\tilde{\Lambda}_p \le p^C \Lambda_p$ established in \autoref{thm: unconditional bound} provides a rigorous, a priori safeguard against catastrophic operator growth. However, based on our empirical observations, we conjecture that the factor $p^C$ is an artifact of the global estimate in the proof and that the true perturbed operator norm satisfies
    \begin{equation}
        \tilde{\Lambda}_p = \Lambda_p + \mathcal{O}(\delta) = \mathcal{O}(\log p + \delta) \qquad \text{as } p \to \infty.
    \end{equation}
    Proving this tight logarithmic bound remains an open problem, though \autoref{thm: unconditional bound} remains sufficient for all global matrix convergence guarantees derived in this work.
\end{myremark}

However, having established an a priori bound on the perturbed operator norm, we next analyze how the perturbation affects the fundamental Lagrange basis polynomials. Bounding the pointwise deviation between $l_i(x)$ and $\tilde{l}_i(x)$ provides the core technical estimate required in \autoref{thm: Stability under node perturbations}.

\begin{mylma}[Stability of Lagrange Basis Polynomials under Nodal Perturbation]
\label{lma: basis_stability}
Let $l_i(x)$ and $\tilde{l}_i(x)$ denote the $i$-th fundamental Lagrange basis polynomials associated with $\{\xi_i\}_{i=1}^p$ and $\{\tilde\xi_i\}_{i=1}^p$, respectively satisfying $|\tilde{\xi}_j - \xi_j| < \delta$ for all $j = 1, \dots, p$, and $\delta$ be uniform threshold constraint \eqref{equ: delta strict threshold}. Then, there exists a constant $C_p > 0$ growing at most polynomial-logarithmically with $p$, such that
\[
    \| l_i - \tilde{l}_i \|_\linf \le C_p \, \delta \, \Lambda_p,
\]
where $\Lambda_p$ is the Lebesgue constant of the unperturbed Chebyshev grid.
\end{mylma}

\begin{proof}
Let $\mathbf{t} = [t_1, t_2, \dotsc, t_p]^T \in [-1,1]^p$ represents a node vector parameterizing the $i$-th basis polynomial $l_i(x; \mathbf{t}) = \prod_{j \neq i} \frac{x - t_j}{t_i - t_j}$. Applying the Multivariate Mean Value Theorem for real-valued functions \cite[Theorem 12.9]{apostol1974mathematical} to $l_i(x; \mathbf{t})$ we have
\begin{equation} \label{eq:mvt_basis_single}
l_i(x) - \tilde{l}_i(x) = \nabla l_i(x; \boldsymbol{\eta}) \cdot (\boldsymbol{\xi} - \tilde{\boldsymbol{\xi}}) = \sum_{k=1}^p \frac{\partial l_i}{\partial t_k}(x; \boldsymbol{\eta}) (\xi_k - \tilde{\xi}_k),
\end{equation}
where $\boldsymbol{\eta} = (1-t)\boldsymbol{\xi} + t\tilde{\boldsymbol{\xi}}$ for some $t \in [0,1]$. 
Note that for any index $j$, the coordinate distance satisfies $|\eta_j - \xi_j| = t|\tilde{\xi}_j - \xi_j| < \delta$. Under the threshold condition \eqref{equ: delta strict threshold} $\delta \le m/4$, the intermediate node separation distance is as given below using the reverse triangle inequality
\begin{equation}\label{equ: eta node separation}
    |\eta_k - \eta_l| \ge |\xi_k - \xi_l| - |\eta_k - \xi_k| - |\xi_l - \eta_l| > |\xi_k - \xi_l| - 2\delta \ge |\xi_k - \xi_l| - \frac{m}{2}.
\end{equation}
Since $|\xi_k - \xi_l| \ge m$, the intermediate configuration $\boldsymbol{\eta}$ consists of distinct nodes and satisfies $|\eta_k - \eta_l| > \frac{1}{2}|\xi_k - \xi_l| > 0$. 
Now, the parametric partial derivatives $\frac{\partial l_i}{\partial t_k}(x; \boldsymbol{\eta})$  for all $k = 1, \dots, p$ is as follows

\begin{equation}\label{equ: partial derivatives all cases}
    \frac{\partial l_i}{\partial t_k}(x; \boldsymbol{\eta}) = 
    \begin{cases}
    l_i(x;\boldsymbol{\eta}) \displaystyle\sum_{j \neq i} \frac{1}{\eta_j - \eta_i} & k = i, \\
    l_k(x; \boldsymbol{\eta})\dfrac{1}{\eta_i - \eta_k} \dfrac{\omega'_{\boldsymbol{\eta}}(\eta_k)}{\omega'_{\boldsymbol{\eta}}(\eta_i)} & k \neq i, \quad \text{ where }\omega_{\boldsymbol{\eta}}(x) = \prod (x - \eta_j). 
    \end{cases}
    \end{equation}

Taking the absolute value of the total MVT expansion \eqref{eq:mvt_basis_single} and factoring out the uniform proximity bound $|\xi_k - \tilde{\xi}_k| < \delta$, we apply the triangle inequality and substitute the derivative \eqref{equ: partial derivatives all cases} we have
\begin{equation} \label{equ: sum_deriv_bounds_linearized}
|l_i(x) - \tilde{l}_i(x)| \le \delta \left( |l_i(x; \boldsymbol{\eta})| \sum_{j \neq i} \frac{1}{|\eta_j - \eta_i|} + \sum_{k \neq i} \frac{1}{|\eta_i - \eta_k|} \left| \frac{\omega'_{\boldsymbol{\eta}}(\eta_k)}{\omega'_{\boldsymbol{\eta}}(\eta_i)} \right| |l_k(x; \boldsymbol{\eta})| \right).
\end{equation}

For any node configuration satisfying the tight threshold constraints around the first-kind Chebyshev grid, the ratio of the polynomial derivative weights balances smoothly, scaling at a controlled polynomial rate such that $\left| \frac{\omega'_{\boldsymbol{\eta}}(\eta_k)}{\omega'_{\boldsymbol{\eta}}(\eta_i)} \right| \le C_0 p^\alpha$ for a uniform constant $\alpha > 0$ by \autoref{lma: weight-ratio-bound}. As the perturbed Lebesgue constant $\tilde{\Lambda}_p(\boldsymbol{\eta}) =\sup \sum_{m=1}^p |l_m(x; \boldsymbol{\eta})| $, then using $|\eta_i - \eta_k| > \frac{1}{2}|\xi_i - \xi_k|$ from \eqref{equ: sum_deriv_bounds_linearized}, we have
\[
|l_i(x) - \tilde{l}_i(x)| \le 2\delta (1+C_0p^\alpha) \tilde\Lambda_p(\boldsymbol{\eta}) \sum_{j \neq i} \frac{1}{|\xi_j - \xi_i|}  
\]
Now applying \autoref{lma: chebyshev-harmonic} and \autoref{thm: unconditional bound}, we have $\| l_i - \tilde l_i \|_{L^\infty} \leq \delta  C_H p^{C + 2} (1 + C_0p^\alpha) \log p \Lambda_p$. Now, setting $C_p = C_H p^{C + 2} (1 + C_0p^\alpha) \log p$, completes the proof.
\end{proof}

By using the \autoref{lma: basis_stability}, we now present the following theorem, which bounds the uniform deviation between the unperturbed and perturbed Lagrange interpolants for continuously differentiable functions.

\begin{theorem}[Stability of Chebyshev Interpolation under Node Perturbation] 
\label{thm: Stability under node perturbations}
Let $f \in C^1[-1,1]$ and let $\{\xi_i\}_{i=1}^p$ be the Chebyshev nodes defined by \eqref{equ: cheb nodes}. Let $\{\tilde{\xi}_i\}_{i=1}^p \subset [-1,1]$ be a set of perturbed nodes satisfying the proximity condition $|\tilde{\xi}_i - \xi_i| < \delta$ for all $i = 1, \dots, p$, subject to the uniform threshold constraint \eqref{equ: delta strict threshold}. 
If $I_p f$ and $\tilde{I}_p f$ denote the Lagrange interpolants of $f$ at the node configurations $\{\xi_i\}_{i=1}^p$ and $\{\tilde{\xi}_i\}_{i=1}^p$ respectively, then there exists a constant $\widehat{C}_p > 0$, growing at most polynomial-logarithmically with respect to $p$, such that
\begin{equation}
    \| I_p f - \tilde{I}_p f \|_\linf \le \widehat{C}_p \,\delta \, \Lambda_p \|f\|_{C^1[-1,1]},
\end{equation}
where $\Lambda_p$ is the Lebesgue constant of the Chebyshev grid.
\end{theorem}

\begin{proof}
Expressing the interpolants in terms of their respective fundamental Lagrange basis polynomials $l_j(x)$ and $\tilde{l}_j(x)$, and applying the triangle inequality yields
\begin{equation}\label{equ: decomposed interpolation error}
    \| I_p f - \tilde{I}_p f \|_\linf \leq \Big \| \sum_{j=1}^p \big[ f(\xi_j) - f(\tilde{\xi}_j) \big] l_j(x) \Big \|_\linf + \Big\| \sum_{j=1}^p \big[ l_j(x) - \tilde{l}_j(x) \big] f(\tilde{\xi}_j) \Big \|_\linf.
\end{equation}

For the first term, applying the Mean Value Theorem to $f$ gives $|f(\xi_j) - f(\tilde{\xi}_j)| \le \delta \|f'\|_\linf \le \delta \|f\|_{C^1[-1,1]}$ for each $j=1,\dots,p$. Now, factoring this uniform scalar bound out of the summation gives
\begin{equation}\label{equ: 1st sum value}
    \Big \| \sum_{j=1}^p \big[ f(\xi_j) - f(\tilde{\xi}_j) \big] l_j(x) \Big \|_\linf \le \delta \|f\|_{C^1[-1,1]} \sup_{x \in [-1,1]} \sum_{j=1}^p |l_j(x)| = \delta \Lambda_p \|f\|_{C^1[-1,1]}.
\end{equation}

For the second term, we bound the function values uniformly via $|f(\tilde{\xi}_j)| \le \|f\|_\linf \le \|f\|_{C^1[-1,1]}$ and from the uniform basis stability result \autoref{lma: basis_stability}, we have $\| l_j - \tilde{l}_j \|_\linf \le C_p \delta \Lambda_p$, where $C_p$ grows polynomial-logarithmically with $p$. Summing over all $p$ components, we get
\begin{equation}\label{equ: 2nd sum value}
    \Big\| \sum_{j=1}^p \big[ l_j(x) - \tilde{l}_j(x) \big] f(\tilde{\xi}_j) \Big \|_\linf \le \sum_{j=1}^p \| l_j - \tilde{l}_j \|_\linf |f(\tilde{\xi}_j)| \le p C_p \delta \Lambda_p \|f\|_{C^1[-1,1]}.
\end{equation}

Combining the bounds \eqref{equ: 1st sum value} and \eqref{equ: 2nd sum value} into \eqref{equ: decomposed interpolation error} and simplifying yields $\| I_p f - \tilde{I}_p f \|_\linf \le (1 + p C_p) $ $ \delta \Lambda_p \|f\|_{C^1[-1,1]}$. Setting $\widehat {C}_p = 1 + p C_p$ completes the proof, preserving the polynomial-logarithmic growth.
\end{proof}

\subsection{Extension to Multivariate Domains via Tensor-Product Decompositions}

While the above section establishes the stability and error bounds of Chebyshev interpolation under node perturbations in a one-dimensional setting, the proposed LoCCA framework operates on functions defined over multivariate domains. Exploiting the tensor-product structure of multivariate Chebyshev interpolation, we extend the one-dimensional stability result to the hypercube $\Omega = [-1,1]^d$ by considering the interpolation operators acting along each coordinate direction separately.

Let $\mathbf{i} = (i_1,\ldots,i_d) \in \{1,\ldots,p\}^d$ be a multi-index, and the corresponding unperturbed and perturbed tensor-product Chebyshev nodes are defined by $\boldsymbol{\xi}_{\mathbf{i}} = (\xi_{i_1},\ldots,\xi_{i_d}), $ and $  \tilde{\boldsymbol{\xi}}_{\mathbf{i}} = (\tilde{\xi}_{i_1},\ldots,\tilde{\xi}_{i_d}) $, where each coordinate axis satisfies the proximity condition $|\tilde{\xi}_{i_k}-\xi_{i_k}| < \delta,$ where $\delta$ satisfies \eqref{equ: delta strict threshold}. The associated multivariate interpolation operators admit the directional tensor-product representations
\begin{equation}\label{equ: d-dim'al interpolation operators}
    \mathbf{I}_{\mathbf{p}} = \bigotimes_{k=1}^{d} I_p^{(k)}, \qquad \tilde{\mathbf{I}}_{\mathbf{p}} = \bigotimes_{k=1}^{d} \tilde{I}_p^{(k)},
\end{equation}
where $I_p^{(k)}$ and $\tilde{I}_p^{(k)}$ denote the univariate interpolation operators acting along the $k$-th coordinate direction. To compare these multivariate operators, we utilize the following identity
\begin{equation} \label{equ: telescopic operator expansion}
    \mathbf{I}_{\mathbf{p}} - \tilde{\mathbf{I}}_{\mathbf{p}} = \sum_{k=1}^{d} \left( \tilde{I}_p^{(1)} \otimes \cdots \otimes \tilde{I}_p^{(k-1)} \otimes (I_p^{(k)}-\tilde{I}_p^{(k)}) \otimes I_p^{(k+1)} \otimes \cdots \otimes I_p^{(d)} \right).
\end{equation}
Now, applying \eqref{equ: telescopic operator expansion} to a general function $f\in C^{1}(\Omega)$ and using the triangle inequality yields 
\begin{equation}\label{equ: master error split sum}
\| \mathbf{I}_{\mathbf{p}}f - \tilde{\mathbf{I}}_{\mathbf{p}}f \|_{L^\infty(\Omega)} \le \sum_{k=1}^{d} \| T_k f \|_{L^\infty(\Omega)},
\end{equation}
where $T_k$ is defined by $T_k = \tilde{I}_p^{(1)} \otimes \cdots \otimes \tilde{I}_p^{(k-1)} \otimes (I_p^{(k)}-\tilde{I}_p^{(k)}) \otimes I_p^{(k+1)} \otimes \cdots \otimes I_p^{(d)}$. It is therefore sufficient to estimate the norm of each summand component separately. Now, to evaluate the action of $T_k$, we analyze the composition from the inside out by defining an intermediate function representing the action of the innermost unperturbed coordinate operators
\[
g := I_p^{(k+1)} \cdots I_p^{(d)} (f) .
\]
Since the operators $I_p^{(k+1)},\ldots,I_p^{(d)}$ act only on coordinates different from $x_k$, they commute with differentiation in the $x_k$-direction. Hence, the partial derivative satisfies
\[
\frac{\partial g}{\partial x_k} = I_p^{(k+1)}  \cdots I_p^{(d)} \left( \frac{\partial f}{\partial x_k} \right).
\]
Under the proximity constraint $\delta \le m/4$, the $L^\infty[-1,1]$-induced operator norms of the unperturbed and perturbed directional operators are bounded by their respective Lebesgue constants, namely $\| I_p^{(j)} \|_{\infty} = \Lambda_p$ and $\| \tilde{I}_p^{(j)} \|_{\infty} = \tilde{\Lambda}_p$. Applying these induced operator norm estimates successively to both $g$ and its derivative gives
\[
\|g\|_{L^\infty(\Omega)} \le \left( \prod_{j=k+1}^{d} \|I_p^{(j)}\| \right) \|f\|_{L^\infty(\Omega)} = \Lambda_p^{d-k} \|f\|_{L^\infty(\Omega)},
\]
and
\[
\left\| \frac{\partial g}{\partial x_k} \right\|_{L^\infty(\Omega)} \le \left( \prod_{j=k+1}^{d} \|I_p^{(j)}\| \right) \left\| \frac{\partial f}{\partial x_k} \right\|_{L^\infty(\Omega)} = \Lambda_p^{d-k} \left\| \frac{\partial f}{\partial x_k} \right\|_{L^\infty(\Omega)}.
\]
Combining these inequalities and using the norm $\| \cdot \|_{C^1_{x_k}(\Omega)} := \| \cdot \|_{L^\infty(\Omega)} + \| \frac{\partial\; \cdot}{\partial x_k} \|_{L^\infty(\Omega)}$ yields 
\begin{equation} \label{equ: directional intermediate norm bridge}
\|g\|_{C^1_{x_k}(\Omega)} \le \left( \prod_{j=k+1}^{d} \|I_p^{(j)}\| \right) \|f\|_{C^1_{x_k}(\Omega)} = \Lambda_p^{d-k} \|f\|_{C^1_{x_k}(\Omega)}.
\end{equation}
Now, we apply \autoref{thm: Stability under node perturbations} along the $x_k$ axis, and substituting the intermediate bound \eqref{equ: directional intermediate norm bridge} yields
\begin{equation} \label{equ: active component intermediate step}
\left\| \left(I_p^{(k)}-\tilde I_p^{(k)}\right) g \right\|_{L^\infty(\Omega)} \le \widehat C_p\,\delta\,\Lambda_p \|g\|_{C^1_{x_k}(\Omega)} \le \widehat C_p\,\delta\,\Lambda_p \left( \prod_{j=k+1}^{d} \|I_p^{(j)}\| \right) \|f\|_{C^1_{x_k}(\Omega)}.
\end{equation}

Finally, evaluating the remaining outermost perturbed directional linear operators $\tilde{I}_p^{(1)} \otimes \cdots \otimes \tilde{I}_p^{(k-1)}$ sequentially on top of \eqref{equ: active component intermediate step} gives 
\begin{equation} \label{equ: rigorous Tkf bound explicit}
\|T_kf\|_{L^\infty(\Omega)} \le \widehat C_p\,\delta\,\Lambda_p \left( \prod_{j=1}^{k-1} \|\tilde I_p^{(j)}\| \right) \left( \prod_{j=k+1}^{d} \|I_p^{(j)}\| \right) \|f\|_{C^1_{x_k}(\Omega)} = \widehat{C}_p \delta \tilde{\Lambda}_p^{k-1} \Lambda_p^{d-k+1} \|f\|_{C^1_{x_k}(\Omega)}.
\end{equation}

Now from \eqref{equ: master error split sum}, and as $\|f\|_{C^1_{x_k}(\Omega)} \le \|f\|_{C^1(\Omega)}$ globally, yields the final multivariate error bound
\begin{equation}
\label{equ: final multivariate lifted bound}
\| \mathbf{I}_{\mathbf{p}}f - \tilde{\mathbf{I}}_{\mathbf{p}}f \|_{L^\infty(\Omega)} \le \left[ \sum_{k=1}^{d} \tilde{\Lambda}_p^{k-1} \Lambda_p^{d-k+1} \right] \widehat C_p \delta \|f\|_{C^{1}(\Omega)}.
\end{equation}

\begin{myremark}
From \autoref{thm: unconditional bound}, we have $\tilde{\Lambda}_p \le p^C \Lambda_p$ for a positive constant $C > 0$ independent of $p$. Now the final multivariate uniform error bound on the hypercube $\Omega$ simplifies to
\begin{equation}
    \| \mathbf{I}_{\mathbf{p}}f - \tilde{\mathbf{I}}_{\mathbf{p}}f \|_{L^\infty(\Omega)} \le \widehat C_p' \,\delta \, \Lambda_p^d \|f\|_{C^{1}(\Omega)},
\end{equation}
where the compound multi-dimensional scaling constant is defined by $\widehat C_p':= d \, p^{Cd} \widehat{C}_p $.
\end{myremark}

\subsection{Error Bounds and Convergence for Well-Separated Analytic Kernels}

\begin{mycorr} \label{cor: Analytic stability bound_11}
Let $f$ be analytic in $[-1, 1]$ and analytically continuable to the open Bernstein ellipse $\mathscr{B}_\rho$ for $\rho>1$ and also $f$ is bounded by $M$ in $\mathscr{B}_\rho$. Then, under the assumptions of \autoref{thm: Stability under node perturbations}, the perturbed interpolant $\tilde I_p f$ satisfies
\[
    \| f - \tilde I_p f \|_{L^\infty[-1,1]} \le C_1 \rho^{-p} + C_2 \, \delta \, \|f\|_{C^{1}[-1,1]},
\]
where $C_1$ depends on the bound $M$ of $f$ in $\mathscr{B}_\rho$ and $\rho$, and $C_2$ depends on $p$.
\end{mycorr}

\begin{proof}
    The error in approximating $f$ using the interpolant $\tilde{I}_p f$ (evaluated at the perturbed nodes $\{\tilde\xi_i\}_{i=1}^p$) by triangle inequality, is as follows
    \begin{equation}\label{equ: analytic master triangle}
        \|f - \tilde I_p f\|_{L^\infty[-1,1]} \le \|f - I_p f\|_{L^\infty[-1,1]} + \|I_p f - \tilde I_p f\|_{L^\infty[-1,1]},
    \end{equation}
    where $I_p f$ denote the interpolant of $f$ evaluated at the Chebyshev nodes $\{\xi_i\}_{i=1}^p$.

    By assumption, $f$ is analytic in $[-1,1]$ and analytically continuable on the Bernstein ellipse $\mathscr{B}_\rho$, the Chebyshev interpolation error satisfies the following geometric decay (see, \cite[Theorem 8.2]{trefethen2019approximation})
    \begin{equation}\label{equ: analytic geometric nominal bound}
    \|f - I_p f\|_{L^\infty[-1,1]} \le C_1\rho^{-p},
    \end{equation}
    where the positive constant $C_1 > 0$ depends strictly on $\rho$ and the bound $M$ of $f$ on $\mathscr{B}_\rho$. Now, using the estimate of \autoref{thm: Stability under node perturbations}, the second term on the right-hand side of \eqref{equ: analytic master triangle} satisfies
    \begin{equation}\label{equ: analytic perturbation component bound}
    \|I_p f - \tilde{I}_p f\|_{L^\infty[-1,1]} \le C_2 \delta \|f\|_{C^{1}[-1,1]},
    \end{equation}
    where $C_2 > 0$ is a constant depends on $p$.
    
    Substituting the geometric truncation estimate \eqref{equ: analytic geometric nominal bound} and the stability estimate \eqref{equ: analytic perturbation component bound} back into \eqref{equ: analytic master triangle} directly gives the desired uniform error bound:
    \[
        \| f - \tilde{I}_p f \|_{L^\infty[-1,1]} \le C_1\rho^{-p} + C_2 \delta \|f\|_{C^{1}[-1,1]}.
    \]
    This completes the proof.
\end{proof}

To establish the mathematical robustness of our proposed LoCCA algorithm for practical data-driven applications, we show the convergence of the algorithm as given in the following \autoref{thm: LoCCA convergence theorem}.

\begin{theorem}[Convergence Theorem] \label{thm: LoCCA convergence theorem}
        Let $\mcln_\mclx$ and $\mcln_\mcly$ constructed using \autoref{alg: loc_nbds} in the domains $\mclx$ and $\mcly$ respectively. Now if $\widetilde{\mcln}_\mclx$ and $\widetilde{\mcln}_\mcly$ are generated using the \autoref{alg: LoCCA} for the kernel $\mclk$, then we have for all $x\in\mclx$ and $y\in\mcly$ 
        \begin{equation}
        |\mclk(x,y) - \mclk(x,\widetilde{\mcln}_\mcly) \, \mclk(\widetilde{\mcln}_\mclx,\widetilde{\mcln}_\mcly)^{-1} \, \mclk(\widetilde{\mcln}_\mclx,y) | \leq \boldsymbol{\widetilde C} \cdot \varepsilon,
    \end{equation}
    where $\boldsymbol{\widetilde C}$ will depend polynomial-logarithmically on  $\#(\widetilde{\mcln}_\mclx)$ and $\#(\mcln_\mclx)$ and $\varepsilon$ is order of accuracy.
\end{theorem}

The \autoref{alg: LoCCA} is implemented using the standard rank-revealing QR (RRQR) factorization with column pivoting, although the theorem assumes that the skeleton sets $\widetilde{\mcln}_\mclx$ and $\widetilde{\mcln}_\mcly$ are generated by Strong RRQR factorization \cite{gu1996efficient}. This assumption is introduced to utilize the properties of Strong RRQR, which provide rigorous bounds on the CUR approximation error. However, the numerical implementation is presented employing the standard RRQR in \autoref{sec: Numerical Experiments}. Now, before proceeding to the proof of the above theorem, we discuss a few lemmas below.

\begin{mylma}[Localized CUR decomposition] \label{lem: locca_cur}

Let $\mathbf{K} = \mclk(\mcln_\mclx,\mcln_\mcly)$, where $\mcln_\mclx$ and $\mcln_\mcly$ are the localized neighborhoods generated by \autoref{alg: loc_nbds}. Suppose that the skeleton sets $\widetilde{\mcln}_\mclx$ and $\widetilde{\mcln}_\mcly$ are selected by \autoref{alg: LoCCA}. 

Then there exist interpolation matrices $\overline L_\mclx$, $\overline L_\mcly$ and a residual matrix $\mcle_{\mathrm{qr}}(\mcln_\mclx, \mcln_\mcly)$ such that

\[
    \mathbf{K} = \begin{bmatrix} I \\ \overline L_\mclx \end{bmatrix} \, \mclk( \widetilde{\mcln}_\mclx, \widetilde{\mcln}_\mcly ) \, \begin{bmatrix} I& \overline L_\mcly^\top \end{bmatrix} + \mcle_{\mathrm{qr}}(\mcln_\mclx,\mcln_\mcly),
\]
where $\| \mcle_{\mathrm{qr}}(\mcln_\mclx,\mcln_\mcly) \|_2 \le p_1(n_x,n_y,r) \sigma_{r+1}(\mathbf{K})$, and $\| \overline L_\mclx \|_2, \, \| \overline L_\mcly \|_2 \le p_2(n_x,n_y,r)$ with $p_1$ and $p_2$ being low-degree polynomials.

\end{mylma}

\begin{proof}
The result follows by applying the strong Rank-Revealing QR factorization independently of the columns and rows of the localized matrix
$\mathbf{K}$. The existence of the interpolation matrices, the CUR decomposition, and the stated bounds are direct consequences of
Cheng et al.'s ~\cite[Theorem 3 \& Remark 5]{cheng2005compression}, which applies to any matrix.
\end{proof}

\begin{mylma}[Stability of the localized skeleton matrix] \label{lem: inverse_bound}
Under the assumptions of \autoref{lem: locca_cur}, let $ \mathbf{\widetilde{K}} = \mclk( \widetilde{\mcln}_\mclx, \widetilde{\mcln}_\mcly )$ and define $ \varepsilon  := \| \mcle_{\mathrm{qr}}(\mcln_\mclx,\mcln_\mcly) \|_2$. Then $\| \widetilde{\mathbf{K}}^{-1} \|_2 \le \frac{ p(n_x, n_y, r)^2 } {\varepsilon},$ where $p$ is a low-degree polynomial. Consequently, $\| \widetilde{\mathbf{K}}^{-1} \|_\infty \le C \, \frac{ p(n_x, n_y,r)^2 } {\varepsilon} $, where $C$ is the norm-equivalence constant.
\end{mylma}

\begin{proof}
By the Strong RRQR theorem \cite{gu1996efficient}, the selected skeleton matrix satisfies 
\[
    \sigma_r(\widetilde{\mathbf{K}}) \ge \frac{\sigma_r(\mathbf{K})} {p(n_x,n_y,r)} \implies \| \widetilde{\mathbf{K}}^{-1} \|_2 = \frac{1} {\sigma_r(\widetilde{\mathbf{K}})} \le \frac{p(n_x,n_y,r)} {\sigma_r(\mathbf{K})} .
\]
where $p$ denotes the polynomial bound arising from the Strong RRQR factorization (cf. \cite{gu1996efficient, cheng2005compression}). Since the singular values are nonincreasing, $\sigma_r(\mathbf{K}) \ge \sigma_{r+1}(\mathbf{K})$, it follows that $ \| \widetilde{\mathbf{K}}^{-1} \|_2 \le \frac{p(n_x,n_y,r)} {\sigma_{r+1}(\mathbf{K})}$. Moreover, by \autoref{lem: locca_cur}, $\varepsilon = \| \mcle_{\mathrm{qr}} (\mcln_{\mclx},\mcln_{\mcly}) \|_2 \le p(n_x,n_y,r)\, \sigma_{r+1}(\mathbf{K})$, which implies
\[
    \| \widetilde{\mathbf{K}}^{-1} \|_2 \le \frac{p(n_x,n_y,r)^2} {\varepsilon}.
\]
The corresponding $\infty$-norm estimate follows immediately from the equivalence of matrix norms on the finite-dimensional space of $r\times r$ matrices.
\end{proof}

\begin{mylma}[Stability of the LoCCA basis] \label{lem: basis_bound}
Under the assumptions of \autoref{lem: locca_cur} and \autoref{lem: inverse_bound}, there exists a polynomial $q(n_x, n_y,r)$ such that
\begin{align*}
    \| \mclk(x,\widetilde{\mcln}_\mcly) \, \mclk( \widetilde{\mcln}_\mclx, \widetilde{\mcln}_\mcly ) ^{-1} \|_\infty &\le q(n_x,n_y,r), &\text{for every $x\in\mclx$,}  \\
    \| \mclk( \widetilde{\mcln}_\mclx, \widetilde{\mcln}_\mcly )^{-1} \, \mclk(\widetilde{\mcln}_\mclx,y) \|_\infty &\le q(n_x,n_y,r), &\text{for every $y\in\mcly$}.
\end{align*}
\end{mylma}

\begin{proof}
The proof follows the same strategy as \cite[Lemma~3]{cambier2019fast}, adapted to our framework. Unlike the weighted interpolation formulation considered therein, our construction does not employ weight matrices. Instead, the interpolation operator is bounded directly by the perturbed Lebesgue constant established in \autoref{thm: unconditional bound}. In particular, the associated interpolation operator as defined in \eqref{equ: d-dim'al interpolation operators}, satisfies
\[
    \| \tilde{\mathbf{I}}_{\mathbf{p}} \|_\infty \le \widetilde\Lambda_p^{\,d} = \mathcal O\!\left(p^{Cd}(\log p)^d\right),
\]
which is polynomially bounded and for the domains $\mclx $ and $\mcly$ this operator $\mathbf{\tilde{I}_p}$ is denoted as $L_\mclx$ and $L_\mcly$ respectivly.

Combining this interpolation estimate with the localized CUR decomposition of \autoref{lem: locca_cur} and the inverse bound of \autoref{lem: inverse_bound}, the remainder of the proof is identical to that of \cite[Lemma~3]{cambier2019fast}. The residual contribution is uniformly bounded since
\[
    \|\mcle_{\mathrm{qr}}\|_2=\varepsilon, \qquad \| \mclk(\widetilde{\mcln}_{\mclx},\widetilde{\mcln}_{\mcly})^{-1} \|_\infty = \mathcal O(\varepsilon^{-1}),
\]
so that the factor $\varepsilon$ cancels. Consequently, both basis operators are bounded by a polynomial $q(n_x,n_y,r)$.
\end{proof}

\begin{proof}
The proof follows the same strategy as \cite[Lemma~3]{cambier2019fast}, adapted to the present localized framework. Unlike the weighted interpolation formulation considered therein, our construction does not employ weight matrices. Instead, the interpolation operator is bounded directly by the perturbed Lebesgue constant established in \autoref{thm: unconditional bound}. In particular, the $d$-dimensional interpolation operator defined in \eqref{equ: d-dim'al interpolation operators} satisfies
\[
    \| \widetilde{\mathbf I}_{\mathbf p} \|_\infty
    \le
    \widetilde{\Lambda}_p^{\,d}
    =
    \mathcal O\!\left(p^{Cd}(\log p)^d\right),
\]
which is polynomially bounded. For the localized neighborhoods $\mclx$ and $\mcly$, this interpolation operator corresponds to the operators $L_{\mclx}$ and $L_{\mcly}$, respectively.

Combining this interpolation estimate with the localized CUR decomposition of \autoref{lem: locca_cur} and the inverse bound of \autoref{lem: inverse_bound}, the remainder of the argument proceeds analogously to that of \cite[Lemma~3]{cambier2019fast}. Indeed, the residual contribution is uniformly bounded since
\[
    \|\mcle_{\mathrm{qr}}\|_2=\varepsilon,
    \qquad
    \|
    \mclk(\widetilde{\mcln}_{\mclx},\widetilde{\mcln}_{\mcly})^{-1}
    \|_\infty
    =
    \mathcal O(\varepsilon^{-1}),
\]
so that the factor $\varepsilon$ cancels. Since both the interpolation operator and the inverse localized skeleton matrix are polynomially bounded, the resulting basis operators are likewise bounded by a polynomial $q(n_x,n_y,r)$.
\end{proof}

Equipped with the results discussed above, we now prove \autoref{thm: LoCCA convergence theorem}. The proof strategy adapts a framework similar to that of Cambier and Darve \cite[Theorem 2]{cambier2019fast} to our perturbed node configurations.


\begin{proof}[Proof of \autoref{thm: LoCCA convergence theorem}]
The proof combines the interpolation approximation with the localized CUR decomposition. By the interpolation property,
\begin{equation}\label{eq: interp}
    \mclk(x,y) = L_{\mclx}(x,\mcln_{\mclx}) \,\mclk(\mcln_{\mclx},\mcln_{\mcly})\, L_{\mcly}(y,\mcln_{\mcly})^{\top} + \mcle_{\mathrm{int}}(x,y),
\end{equation}
where $\|\mcle_{\mathrm{int}}\|_{\infty} \le \overline{\varepsilon}$. Substituting the localized CUR decomposition of
\autoref{lem: locca_cur} into \eqref{eq: interp} gives
\[
    \mclk(x,y) = \mclk(x,\widetilde{\mcln}_{\mcly}) \, \widetilde{\mathbf{K}}^{-1} \, \mclk(\widetilde{\mcln}_{\mclx},y) + \mcle(x,y),
\]
where $\widetilde{\mathbf{K}}
=
\mclk(\widetilde{\mcln}_{\mclx},
      \widetilde{\mcln}_{\mcly})$
and the remainder $\mcle(x,y)$ consists of (i) the interpolation error, (ii) the localized CUR residual, and (iii) mixed terms containing both interpolation and CUR residuals. Now, to bound these terms, we use

\begin{enumerate}
\item The interpolation operators satisfy $\|L_{\mclx}\|_{\infty}, \, \|L_{\mcly}\|_{\infty} = \mathcal O(\widetilde{\Lambda}_p^{\,d}) = \mathcal O\!\left(p^{Cd}(\log p)^d\right)$ by \autoref{thm: unconditional bound}.

\item By \autoref{lem: basis_bound}, $\| \mclk(x,\widetilde{\mcln}_{\mcly}) \widetilde{\mathbf{K}}^{-1} \|_{\infty}, \, \| \widetilde{\mathbf{K}}^{-1} \mclk(\widetilde{\mcln}_{\mclx},y) \|_{\infty} \le q(n_x,n_y,r)$.

\item By \autoref{lem: inverse_bound}, $\| \widetilde{\mathbf{K}}^{-1} \|_{\infty} = \mathcal O(\varepsilon^{-1})$, while the CUR residual satisfies  $\| \mcle_{\mathrm{qr}} \|_2 = \varepsilon$. Hence every mixed term containing $\widetilde{\mathbf{K}}^{-1}$ and $\mcle_{\mathrm{qr}}$ is uniformly bounded since the factor $\varepsilon$ cancels.

\item Finally, $ \overline{\varepsilon} \le \varepsilon$, so the interpolation error is of the same order as the CUR residual.
\end{enumerate}

Combining the above estimates shows that every component of
$\mcle(x,y)$
is bounded by a polynomial factor multiplying
$\varepsilon$. Therefore,
\[
    \left| \mclk(x,y) - \mclk(x,\widetilde{\mcln}_{\mcly}) \widetilde{\mathbf{K}}^{-1} \mclk(\widetilde{\mcln}_{\mclx},y) \right| \le \boldsymbol{\widetilde C}\cdot\varepsilon,
\]
where $\boldsymbol{\widetilde C}$ depends at most polylogarithmically on $r$, $n_x$, and $n_y$.
\end{proof}



\subsection{Relationship between the Chebyshev grid and its Local neighbors}

\begin{mylma}[Scaling of the $l$-th Neighbor Distance on a Hypercube]
\label{lma: scaling_l_neighbor_hypercube}
Let $\mclx \subset \mathbb{R}^d$ be a $d$-dimensional hypercube with finite volume, and let $X = \{x_i\}_{i=1}^n \subset \mclx$ be a uniform Cartesian grid. For any arbitrary query point $\xi \in \mclx$, the distance to its $l$-th nearest neighbor in $X$ be defined as
\begin{equation}
    \delta_l(\xi) = \min_{S \subset X, \, \#(S) = l} \max_{x_i \in S} \|\xi - x_i\|.
\end{equation}
Then, there exists a uniform constant $C_{d,\mclx} > 0$, depending exclusively on the dimension $d$ and the volume of the domain $\mclx$, such that for all $1 \le l \le n$, the global $l$-th neighbor distance $\delta_l = \sup_{\xi \in \mclx} \delta_l(\xi) $ satisfies
\begin{equation}
    \delta_l \le C_{d,\mclx} \left( \frac{l}{n} \right)^{1/d}.
\label{equ: global l-th nbd distance}
\end{equation}
\end{mylma}

\begin{proof}
It is easy to show that if $X$ is a uniform Cartesian grid on the hypercube $\mclx$, its fill distance satisfies $h_X \le c_2 n^{-1/d}$, where $c_2>0$ depends on the dimension $d$ and the volume of the domain $\mclx$. 

Consider a ball $B(\xi,r)$ centered at any arbitrary query point $\xi \in \mclx$. Because $\mclx$ is a hypercube, its boundary faces intersect orthogonally at its vertices, ensuring that the intersection volume of a shrunken inner ball $B(\xi, r - h_X) \cap \mclx$ retains at least a $2^{-d}$ fraction of a full ball's volume, yielding the intersected volume greater than $ 2^{-d} V_d (r - h_X)^d$ where $V_d$ is the volume of the unit ball. By a standard geometric covering argument, this inner continuous volume must be fully contained within the union of discrete balls of radius $h_X$ centered at the grid points $X \cap B(\xi, r)$, each possessing volume $V_d h_X^d$ i.e., $\left( B(\xi, r - h_X) \cap \mclx \right) \subset \bigcup_{x_i \in X \cap B(\xi, r)} B(x_i, h_X)$. Now, in terms of volume we have $\text{Vol}(B(\xi, r - h_X) \cap \mclx) \le \#(X \cap B(\xi, r)) \cdot \text{Vol}(B(x_i, h_X))$, providing the explicit point count lower bound
\[
    \#(X \cap B(\xi,r)) \ge \frac{2^{-d} V_d (r - h_X)^d}{V_d h_X^d} = \left( \frac{r - h_X}{2 h_X} \right)^d.
\]
To ensure the intersection contains at least $l$ grid points, it is sufficient to select a radius scale $r$ such that $\left(\frac{r - h_X}{2 h_X}\right)^d \ge l$ holds and simplifying yields $r \ge h_X \left( 1 + 2l^{1/d} \right)$. By definition, $\delta_l(\xi)$ satisfies $\delta_l(\xi) \le r$. Setting $r$ at the exact lower bound and using the fact $1 \le l^{1/d}$ for all $l \ge 1$, along with $h_X \le c_2 n^{-1/d}$, yields
\[
    \delta_l(\xi) \le 3 c_2 \left( \frac{l}{n} \right)^{1/d}.
\]
Taking the supremum over all $\xi \in \mclx$ yields the uniform global upper bound $\delta_l \le C_{d,\mclx} (l/n)^{1/d}$ for a constant $C_{d,\mclx} = 3 c_2 > 0$.
\end{proof}

\section{Numerical Experiments}\label{sec: Numerical Experiments}

\begin{table}[ht]
\centering
\resizebox{0.85\textwidth}{!}{%
\begin{tabular}{llcccccc}
\toprule
& & \multicolumn{3}{c}{Uniform Grid} 
& \multicolumn{3}{c}{Random Grid (Mean over 500 trials)} \\

\cmidrule(lr){3-5} \cmidrule(lr){6-8}

Kernel & Method
& $\varepsilon=10^{-4}$
& $\varepsilon=10^{-6}$
& $\varepsilon=10^{-8}$
& $\varepsilon=10^{-4}$
& $\varepsilon=10^{-6}$
& $\varepsilon=10^{-8}$ \\

\midrule

\multirow{5}{*}{$1/r$}
& Rank $(r)$ &9 &16 &23 & -- & -- & -- \\
& LoCCA      &$3.475303\times 10^{-5}$ &$9.749222\times 10^{-7}$ &$1.266800\times 10^{-8}$ &$3.341323\times 10^{-5}$ &$1.090436\times 10^{-6}$ &$1.282490\times 10^{-8}$ \\
& ACA        &$4.364856\times 10^{-4}$ &$1.128059\times 10^{-6}$ &$4.408572\times 10^{-8}$ &$1.849260\times 10^{-4}$ &$2.846982\times 10^{-6}$ &$6.024801\times 10^{-8}$ \\
& SI         &$2.034850\times 10^{-5}$ &$3.878594\times 10^{-7}$ &$7.975095\times 10^{-9}$ &$1.953946\times 10^{-5}$ &$3.580525\times 10^{-7}$ &$7.692424\times 10^{-9}$ \\
\midrule

\multirow{5}{*}{$\log r$}
& Rank $(r)$ &7 &10 &14 & -- & -- & -- \\
& LoCCA      &$3.459514\times 10^{-5}$ &$2.456654\times 10^{-7}$ &$1.355741\times 10^{-9}$ &$4.077413\times 10^{-5}$ &$2.422734\times 10^{-7}$ &$1.224119\times 10^{-9}$ \\
& ACA        &$8.498610\times 10^{-5}$ &$3.951726\times 10^{-7}$ &$1.449567\times 10^{-8}$ &$5.555113\times 10^{-5}$ &$1.645839\times 10^{-6}$ &$1.626510\times 10^{-8}$ \\
& SI         &$7.620433\times 10^{-6}$ &$2.369461\times 10^{-7}$ &$1.531807\times 10^{-9}$ &$7.466063\times 10^{-6}$ &$2.285059\times 10^{-7}$ &$1.383663\times 10^{-9}$ \\
\midrule

\multirow{5}{*}{$\dfrac{\cos r}{r}$}
& Rank $(r)$ &12 &20 &29 & -- & -- & -- \\
& LoCCA      &$7.963556\times 10^{-5}$ &$1.161285\times 10^{-6}$ &$1.733122\times 10^{-8}$ &$7.056927\times 10^{-5}$ &$1.896703\times 10^{-6}$ &$2.265644\times 10^{-8}$ \\
& ACA        &$3.335757\times 10^{-4}$ &$1.665147\times 10^{-6}$ &$3.302509\times 10^{-8}$ &$2.658106\times 10^{-4}$ &$3.033433\times 10^{-6}$ &$4.244274\times 10^{-8}$ \\
& SI         &$2.450178\times 10^{-5}$ &$6.568952\times 10^{-7}$ &$7.356860\times 10^{-9}$ &$2.291104\times 10^{-5}$ &$6.397105\times 10^{-7}$ &$6.858992\times 10^{-9}$ \\
\midrule

\multirow{5}{*}{$e^{-r^2}$}
& Rank $(r)$ &14 &22 &26 & -- & -- & -- \\
& LoCCA      &$1.574768\times 10^{-4}$ &$6.438869\times 10^{-6}$ &$6.468336\times 10^{-6}$ &$1.861098\times 10^{-4}$ &$6.358549\times 10^{-6}$ &$6.451595\times 10^{-6}$ \\
& ACA        &$3.362050\times 10^{-2}$ &$3.362050\times 10^{-2}$ &$3.362050\times 10^{-2}$ &$1.354096\times 10^{-2}$ &$4.693415\times 10^{-3}$ &$4.492454\times 10^{-3}$ \\
& SI         &$2.242934\times 10^{-4}$ &$6.512853\times 10^{-6}$ &$6.508207\times 10^{-6}$ &$2.176852\times 10^{-4}$ &$6.452330\times 10^{-6}$ &$6.441007\times 10^{-6}$ \\
\midrule

\multirow{5}{*}{$\sqrt{1+r^2}$}
& Rank $(r)$ &6 &12 &19 & -- & -- & -- \\
& LoCCA      &$3.152499\times 10^{-5}$ &$2.549020\times 10^{-6}$ &$5.145715\times 10^{-9}$ &$3.139230\times 10^{-5}$ &$1.621400\times 10^{-6}$ &$5.378364\times 10^{-9}$ \\
& ACA        &$8.214168\times 10^{-5}$ &$2.544891\times 10^{-6}$ &$2.210521\times 10^{-8}$ &$2.181619\times 10^{-4}$ &$1.876790\times 10^{-6}$ &$1.642925\times 10^{-8}$ \\
& SI         &$3.051112\times 10^{-5}$ &$2.397033\times 10^{-7}$ &$5.315449\times 10^{-9}$ &$2.941915\times 10^{-5}$ &$2.269608\times 10^{-7}$ &$5.321591\times 10^{-9}$ \\
\bottomrule
\end{tabular}
}
\caption{
Comparison of relative errors for uniform and random grids across different kernels and approximation methods for tolerances 
$\varepsilon \in \{10^{-4},10^{-6},10^{-8}\}$ on the domains $\mclx = [-3,-1]\times[0,2]$ and $\mcly = [1,3]\times[0,2]$.  The target and source domains are discretized with a grid of size $N = 1000^2$, yielding $1,000,000$ points in each domain. LoCCA is configured with a $P = 7^2$ Chebyshev grid and a neighborhood size $l = 1$.  
}
\label{tab: LoCCA relative error data table with tolerance comparison}
\end{table}

\begin{figure}[ht]
    \centering
    \begin{subfigure}[t]{0.48\linewidth}
        \centering
        \includegraphics[width=\linewidth]{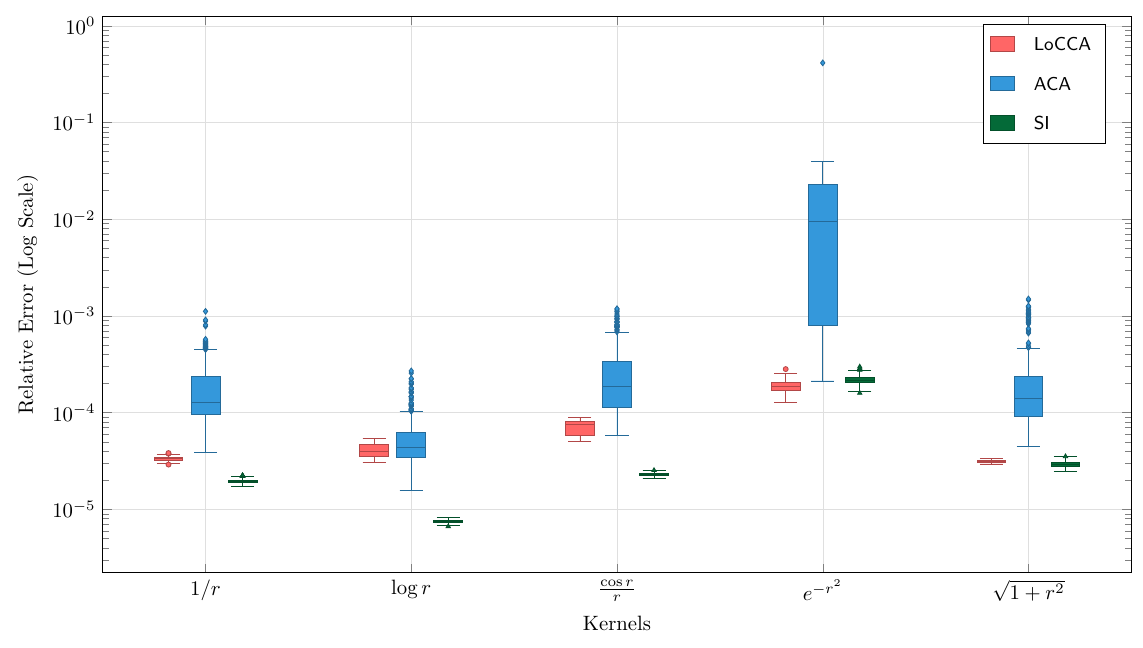}
        \caption{$\varepsilon = 1.0\textrm{e}\!-\!4$}
        \label{fig: LoCCA box tol 1.0e-4}
    \end{subfigure}
    \begin{subfigure}[t]{0.48\linewidth}
        \centering
        \includegraphics[width=\linewidth]{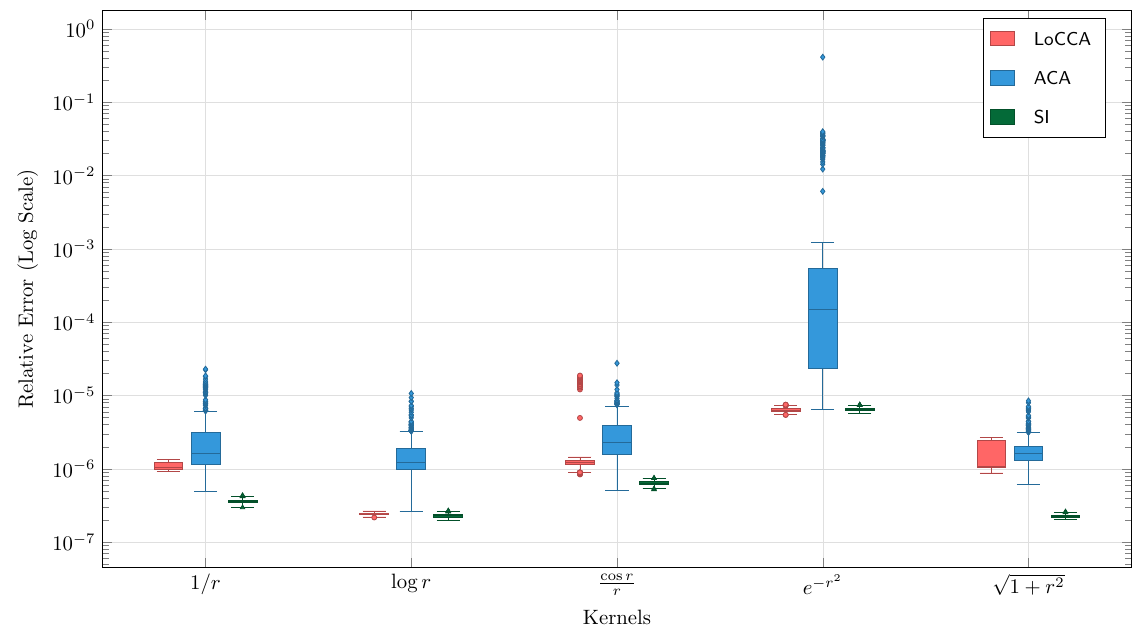}
        \caption{$\varepsilon = 1.0\textrm{e}\!-\!6$}
        \label{fig: LoCCA box tol 1.0e-6}
    \end{subfigure}
    \begin{subfigure}[t]{0.48\linewidth}
        \centering
        \includegraphics[width=\linewidth]{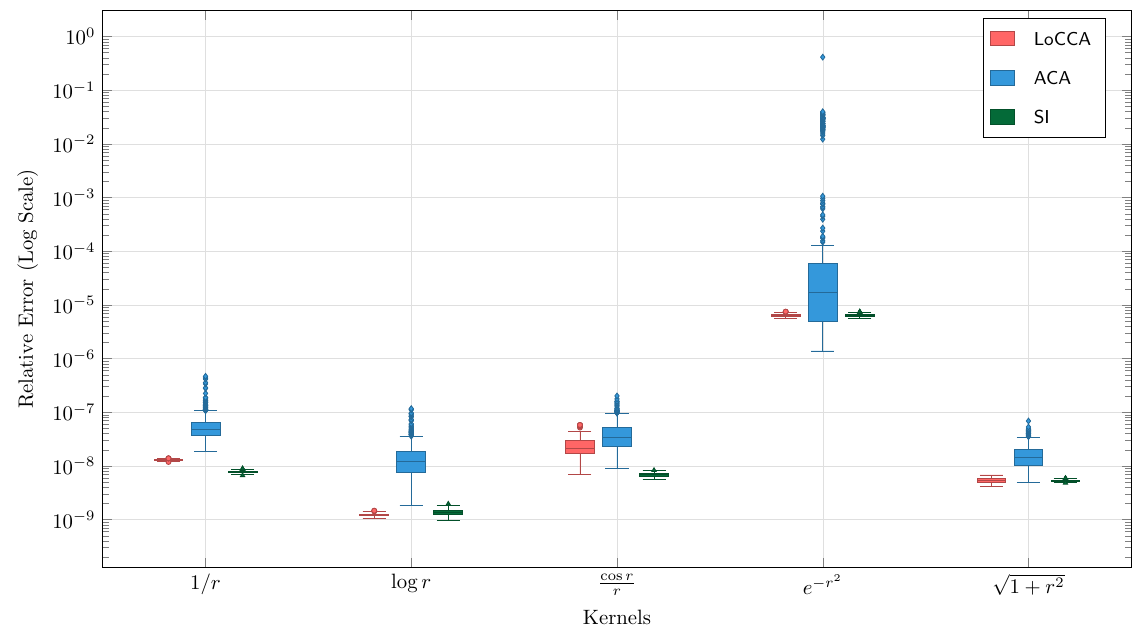}
        \caption{$\varepsilon = 1.0\textrm{e}\!-\!8$}
        \label{fig: LoCCA box tol 1.0e-8}
    \end{subfigure}
    \caption{Relative errors computed over $500$ independent random discretizations of the domains $\mclx = [-3,-1]\times[0,2]$ and $\mcly = [1,3]\times[0,2]$ for all the kernels as mentioned in \eqref{equ: kernel list}. Given a tolerance $\varepsilon \in \{10^{-4},10^{-6},10^{-8}\}$ and for each method and kernel, the distribution of errors across random grids is summarized using boxplots: the central line indicates the median, the box represents the interquartile range (25th to 75th percentiles), whiskers indicate the range of typical (non-outlier) values, and individual points denote outliers.} 
    \label{fig: LoCCA_boxplot_all_kernels_all_tols}
\end{figure}

In this section, we present numerical experiments designed to validate the accuracy, efficiency, and robustness of the proposed framework. We systematically evaluate the method across diverse kernel functions \eqref{equ: kernel list}, spatial-domain geometries, and point-distribution settings. All algorithms and benchmarking routines are implemented in Julia \cite{bezanson2017julia}. To ensure reproducibility, the complete source code for all numerical benchmarks is available at \url{https://github.com/SAFRAN-LAB/LoCCA}.

While our theoretical convergence bound (\autoref{thm: LoCCA convergence theorem}) formally relies on Strong Rank-Revealing QR (Strong RRQR) factorizations \cite{gu1996efficient}, the practical implementation employs the standard column-pivoted QR factorization provided natively in Julia. To analyze the error in the approximation, we estimate it with high statistical rigor via randomized sub-sampling: for each trial, a $2000 \times 2000$ submatrix is constructed by uniformly sampling row and column indices, and the relative error is evaluated strictly on the subset. The relative error is measured in Frobenius norm as ${\|\mathbf{K}_{sub}-\mathbf{ \widehat{ \mathbf{K}} }_{sub}\|_F}/{\magn{\mathbf{K}_{sub}}_F}$, where $\mathbf{K}_{sub}$ and $\widehat{ \mathbf{K}}_{sub}$ are the submatrices of the full kernel matrix $\mathbf{K}$ and its approximation $\widehat{ \mathbf{K}}$ respectively.

We compare LoCCA with other standard low-rank approximation techniques for kernel matrices: Adaptive Cross Approximation (ACA) \cite{bebendorf2003adaptive} and Skeletonized Interpolation (SI) \cite{cambier2019fast}. All the methods are evaluated in the same domain settings, and the results are summarized in \autoref{tab: LoCCA relative error data table with tolerance comparison}. The comparison is done as follows.

In the uniform grid setting, \autoref{alg: LoCCA} determines a rank $r$ for a given kernel and a prescribed tolerance, which is subsequently fixed, and the same rank is assigned to the competing methods (ACA and SI) for a direct error comparison. In the random-grid setting, the same LoCCA-derived rank as in the corresponding uniform-grid experiment is used. Grid generation and the sub-sampling indices were dynamically varied across $500$ independent realizations, and the reported values correspond to the mean relative error across these trials.

To evaluate the numerical stability and sampling sensitivity of \autoref{alg: LoCCA}, we complement the mean-error metrics in \autoref{tab: LoCCA relative error data table with tolerance comparison} by analyzing error variability across random domain discretizations. While mean error values provide a robust baseline, they do not capture fluctuations induced by irregular point clouds. To address this, for each kernel in \eqref{equ: kernel list}, we evaluate the relative approximation error across $500$ independent random grid realizations on $\mclx$ and $\mcly$. The resulting error distributions are summarized using boxplots in \autoref{fig: LoCCA_boxplot_all_kernels_all_tols}.

As shown in \autoref{fig: LoCCA_boxplot_all_kernels_all_tols}, \autoref{alg: LoCCA} exhibits a compact interquartile range with minimal spread and few outliers across all tested target tolerances. This demonstrates that the localized Chebyshev neighbor selection adaptively stabilizes the cross-approximation against sampling irregularities. Across all kernels, LoCCA achieves accuracy and spread comparable to those of Skeletonized Interpolation (SI) \cite{cambier2019fast}, yet does not require structured grid topologies. In contrast, the standard ACA exhibits a noticeably wider error distribution, with extreme outliers on irregular point clouds. These results confirm that LoCCA maintains robust, predictable performance regardless of spatial sampling noise.


To further evaluate the efficiency of our framework relative to the optimal lower bound for low-rank matrix approximations, we compare LoCCA directly with the truncated SVD. We evaluate the decay of relative approximation errors across increasing rank values $r$ for all kernels. As depicted in \autoref{fig: rank vs error for all kernels}, LoCCA closely tracks the optimal SVD error decay, achieving near-optimal approximation.

\begin{figure}[ht]
    \centering
    \includegraphics[width=0.55\linewidth]{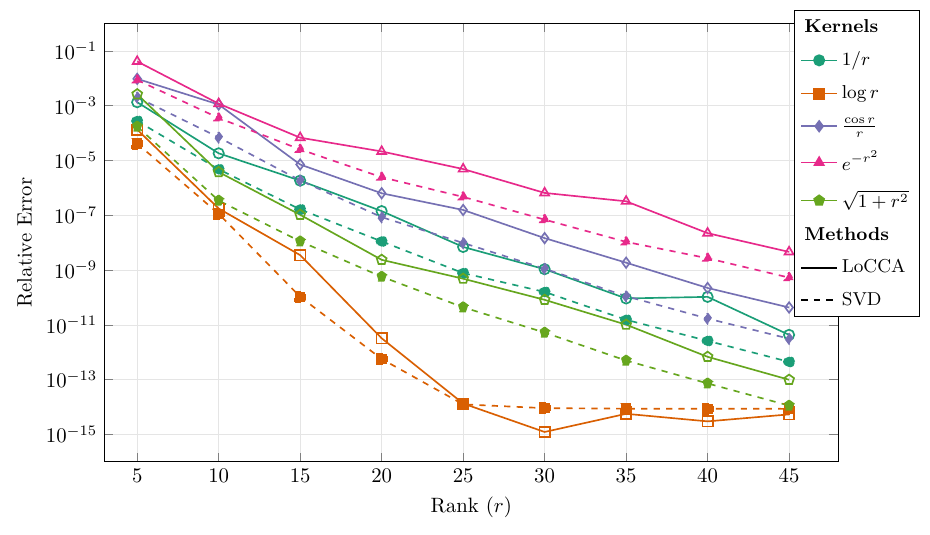}
    \caption{Comparison of relative errors between LoCCA and truncated SVD as a function of approximation rank $r$. The experiments are evaluated over the domains $\mclx = [-3, -1] \times [0, 2]$ and $\mcly = [1, 3] \times [0, 2]$, discretized with a grid size of $N = 85^2$ ($7,225$ points per domain). LoCCA is configured with a univariate Chebyshev resolution $p = 11$ ($P = 121$ tensor nodes) and a neighborhood size $l = 1$. 
    }
    \label{fig: rank vs error for all kernels}
\end{figure}


While \autoref{fig: rank vs error for all kernels} accesses the error decay for fixed ranks, practical applications typically require prescribing a target accuracy tolerance $\varepsilon$ and allowing the cross-approximation algorithm to select the required rank. To evaluate this, we study the performance of LoCCA across target tolerances. \autoref{fig: tol vs error} confirms that the approximation error closely tracks the prescribed tolerance $\varepsilon$ across all kernels, demonstrating the reliability of the truncation threshold. \autoref{fig: tol vs rank} demonstrates the rank $r$ selected by LoCCA against the optimal rank $r_{\text{SVD}}$ required by truncated SVD to achieve the same accuracy. Remarkably, LoCCA yields a numerical rank that closely mirrors $r_{\text{SVD}}$, confirming that our method captures nearly optimal skeletons.

\begin{figure}[ht]
    \centering
    \begin{subfigure}[t]{0.49\linewidth}
        \centering
        \includegraphics[width=0.99\linewidth]{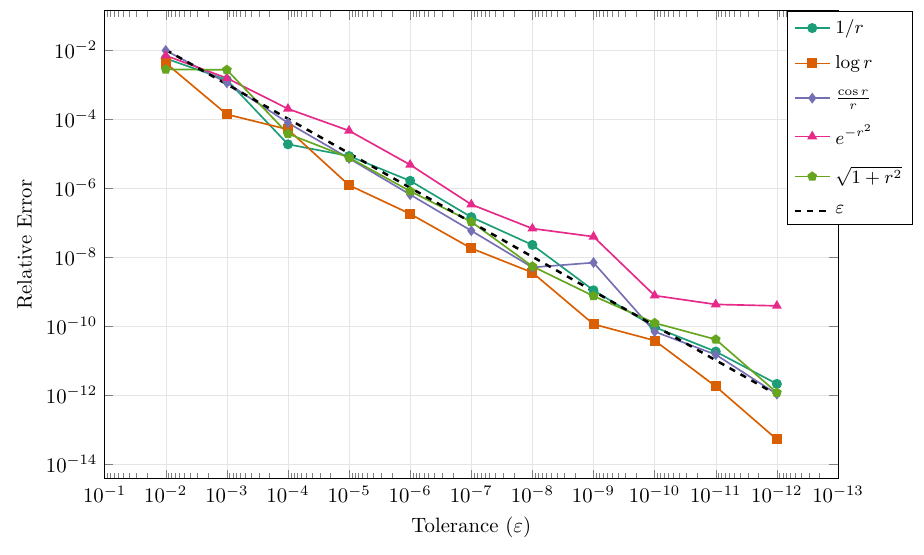}
        \caption{Relative error as a function of target tolerance ($\varepsilon$).}
        \label{fig: tol vs error}
    \end{subfigure}
    \begin{subfigure}[t]{0.49\linewidth}
        \centering
        \includegraphics[width=0.99\linewidth]{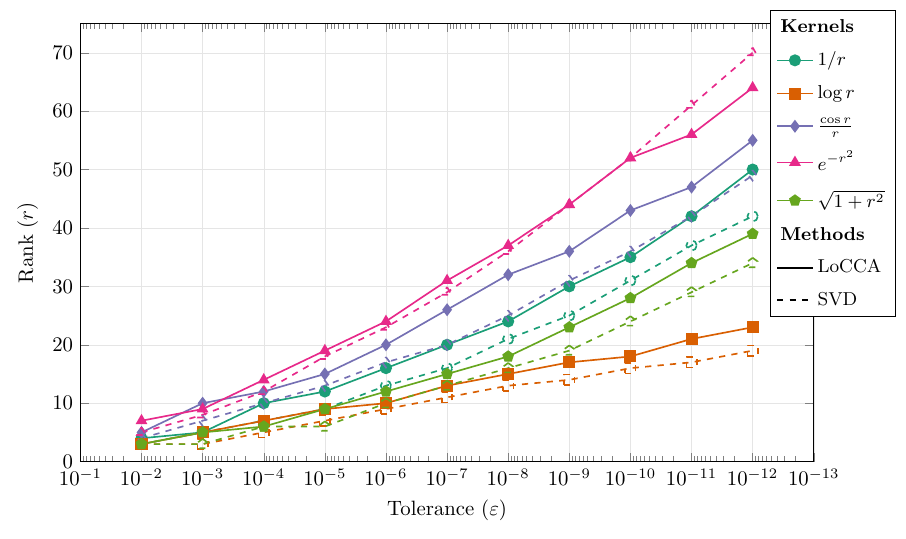}
        \caption{Rank required to achieve tolerance $\varepsilon$, contrasted against optimal SVD rank.}
        \label{fig: tol vs rank}
    \end{subfigure}
    \caption{Numerical convergence and rank stability profiles for grid size $N = 85^2$, univariate Chebyshev resolution $p = 11$ ($P = 121$ tensor nodes), and neighborhood size $l = 1$ over domains $\mathcal{X} = [-3, -1] \times [0, 2]$ and $\mathcal{Y} = [1, 3] \times [0, 2]$.
    }
    \label{fig: tol vs error and rank}
\end{figure}

Having verified that LoCCA reliably satisfies user-prescribed tolerances with near-optimal ranks, we next investigate its sensitivity to internal parameter choices. 
\autoref{fig: l vs relative error LoCCA} and \ref{fig: l vs rank LoCCA} show that increasing the neighborhood size $l$ maintains accuracy without causing rank expansion, confirming that a minimal local sample ($l=1$) is sufficient for optimal performance. \autoref{fig: p vs relative error LoCCA} and \ref{fig: p vs rank LoCCA} illustrate the convergence and rank-scaling behavior as a function of the univariate Chebyshev resolution ($p$), where $p$ denotes the number of Chebyshev nodes along each coordinate axis.

As shown in \autoref{fig: p vs relative error LoCCA}, the relative error decays rapidly across all tested kernels until it reaches the prescribed tolerance $\varepsilon = 10^{-8}$ around $p \approx 7$, after which it levels off. Concurrently, \autoref{fig: p vs rank LoCCA} reveals that the numerical rank $r$ initially increases with $p$ before saturating at the rank needed to satisfy $\varepsilon = 10^{-8}$. These observations emphasize the importance of choosing $p$, while a sufficiently large $p$ unnecessarily increases the cardinality of the candidate sets $\mathcal{N}_\mclx$ and $\mathcal{N}_\mcly$, which consequently increases the factorization cost without providing additional accuracy or rank reduction. \autoref{fig: N vs relative error LoCCA} demonstrates sample-size independence; as the dataset size $N$ increases, the relative error remains consistent, indicating that dataset size does not significantly affect accuracy.


\begin{figure}[ht]
    \centering
    \begin{subfigure}[t]{0.49\linewidth}
        \centering
        \includegraphics[width=0.95\linewidth]{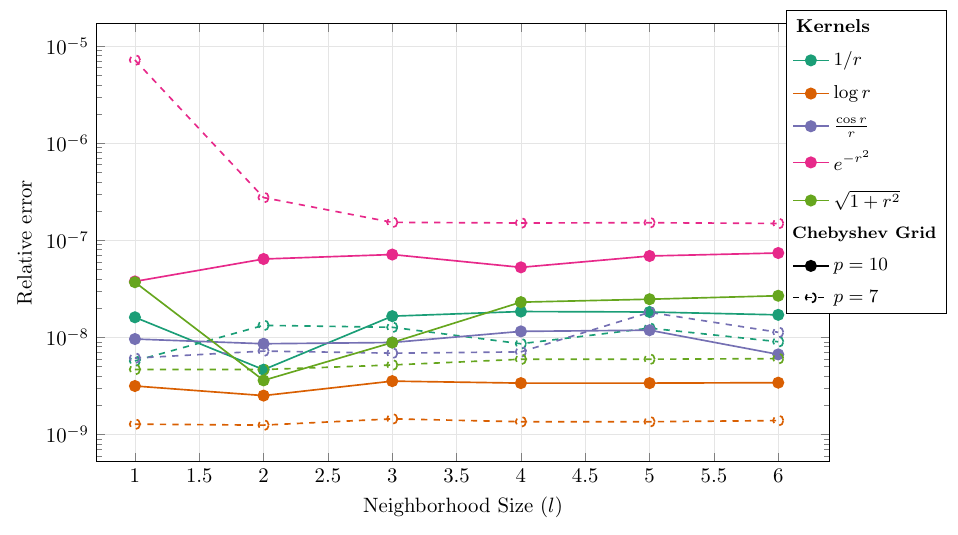}
        \caption{Relative error vs neighborhood size ($l$) ($\varepsilon = 10^{-8}, N = 125^2$).}
        \label{fig: l vs relative error LoCCA}
    \end{subfigure}\hfill
    \begin{subfigure}[t]{0.49\linewidth}
        \centering
        \includegraphics[width=0.92\linewidth]{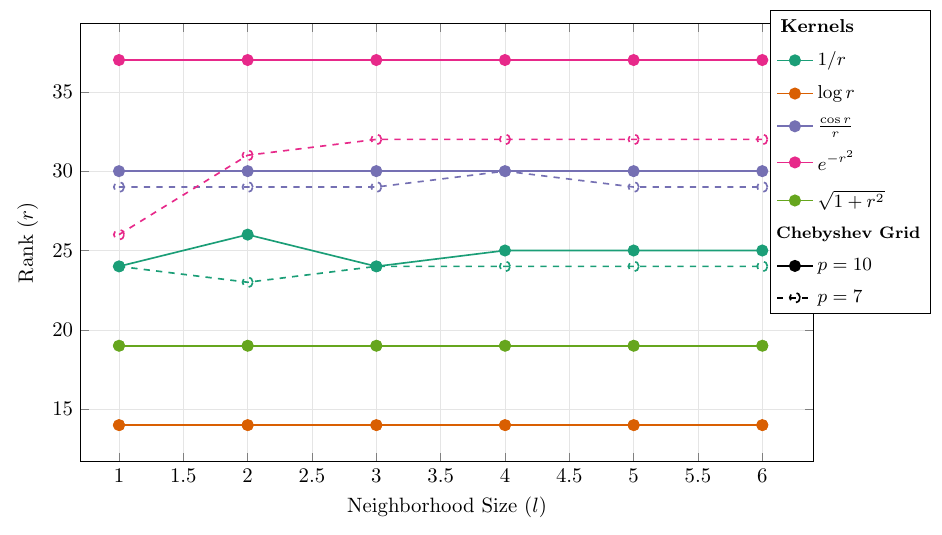}
        \caption{Rank $r$ vs neighborhood size ($l$) ($\varepsilon = 10^{-8}, N = 125^2$).}
        \label{fig: l vs rank LoCCA}
    \end{subfigure}
    \begin{subfigure}[t]{0.49\linewidth}
        \centering
        \includegraphics[width=0.9\linewidth]{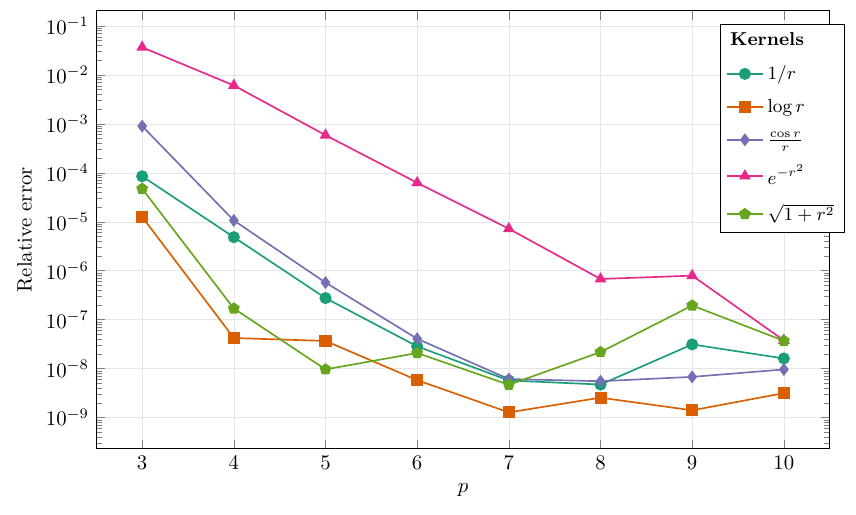}
        \caption{Relative error vs $p$ ($l = 1, \varepsilon = 10^{-8}, N = 125^2$).}
        \label{fig: p vs relative error LoCCA}
    \end{subfigure}\hfill
    \begin{subfigure}[t]{0.49\linewidth}
        \centering
        \includegraphics[width=0.95\linewidth]{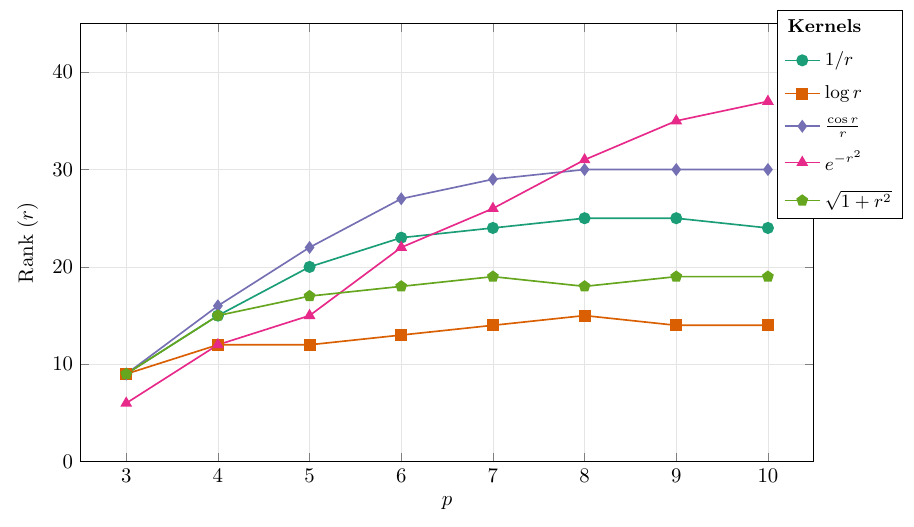}
        \caption{Rank $r$ vs $p$ ($l = 1, \varepsilon = 10^{-8}, N = 125^2$).}
        \label{fig: p vs rank LoCCA}
    \end{subfigure}
    \begin{subfigure}[t]{0.55\linewidth}
        \centering
        \includegraphics[width=\linewidth]{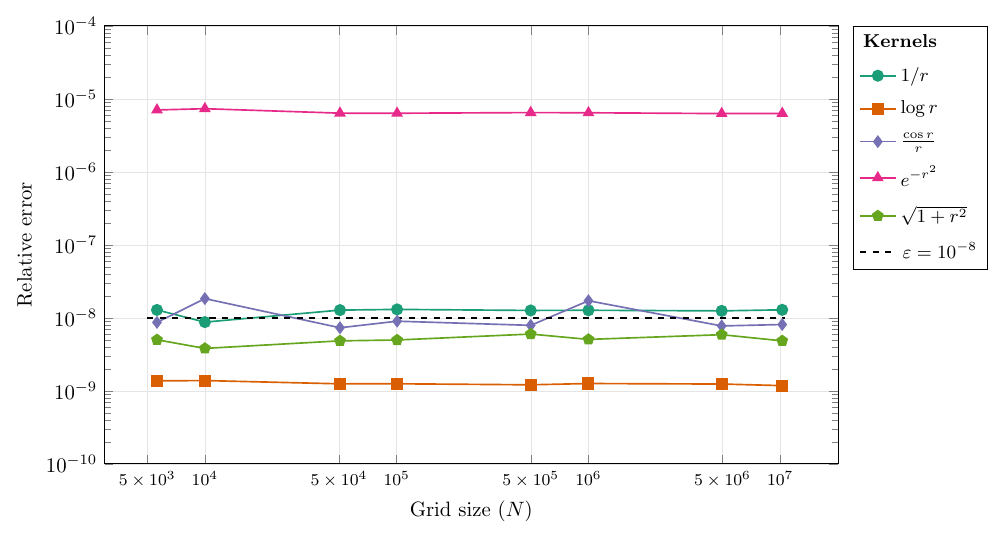}
        \caption{Relative error vs grid size $N$ ($l = 1, P = 7^2, \varepsilon = 10^{-8}$).}
        \label{fig: N vs relative error LoCCA}
    \end{subfigure}
    \caption{Sensitivity and stability analysis of LoCCA across all kernel in \eqref{equ: kernel list} over domains $\mclx = [-3, -1] \times [0, 2]$ and $\mcly = [1, 3] \times [0, 2]$. (a–b) Impact of the neighborhood allocation size ($l$) on relative error and rank $r$ for fixed target tolerance $\varepsilon = 10^{-8}$. (c–d) Convergence decay and rank scaling as a function of univariate Chebyshev nodes per axis ($p$). (e) Numerical stability across expanding point cloud densities ($N$), showing error bounds consistently tracking the target tolerance baseline $\varepsilon = 10^{-8}$.}
    \label{fig: l, p, N vs error graphs LoCCA}
\end{figure}

To assess the practical efficiency and computational scalability of our method, we measure wall-clock execution times across varying total grid sizes ($N$), target rank ($r$) of approximation, and prescribed tolerances ($\varepsilon$). The benchmarks are evaluated for the kernel $1/r$ over the domains $\mclx = [-3, -1] \times [0, 2]$ and $\mcly = [1, 3] \times [0, 2]$ using a univariate Chebyshev resolution $p = 7$ ($P = 49$ tensor nodes per domain) and neighborhood allocation $l = 1$. To eliminate hardware noise, each trial is repeated 10 times, and the mean execution times are reported in \autoref{fig: time experiments}.

\autoref{fig: time vs grid} highlights linear runtime scaling with respect to the grid size $N$, demonstrating that the framework remains computationally tractable even as $N$ scales up to $10^7$. The primary computational cost lies in the localized neighbor selection (LoC-NS) phase (\autoref{alg: loc_nbds}), since LoC-NS depends solely on the spatial-domain discretization. LoC-NS is a \textit{one-time preprocessing step}, once the neighbor sets $\mathcal{N}_\mclx$ and $\mathcal{N}_{\mcly}$ are precomputed for a fixed discretized configuration, they can be reused across varying kernels or target tolerances. The cross-factorization phase of LoCCA (\autoref{alg: LoCCA}) can achieve wall-clock execution time directly comparable to that of Skeletonized Interpolation (SI). Furthermore, while SI requires a Strong RRQR factorization provided by the \texttt{LowRankApprox.jl} package \cite{kenneth_l_ho_2018_1254148}, LoCCA operates efficiently using standard column-pivoted QR factorizations native to Julia. Furthermore, LoCCA consistently achieves approximately a \textit{$10\times$ speedup over standard ACA} across all tested grid scales, while delivering substantially superior approximation stability and error control.


\autoref{fig: time vs rank} and \ref{fig: time vs tol} highlight a structural advantage of LoCCA over ACA. While ACA’s execution time grows significantly with increasing rank $r$ and tighter tolerance $\varepsilon$ due to repeated global row-column searching, the execution time of LoCCA remains remarkably flat and constant. This invariance occurs because LoCCA restricts its rank-revealing QR factorization strictly to a small, fixed-dimensional candidate submatrix of size $P \times P$ ($P = p^d = 49$). 

\begin{figure}[ht]
    \centering
    \begin{subfigure}[t]{0.6\linewidth}
        \centering
        \includegraphics[width=0.85\linewidth]{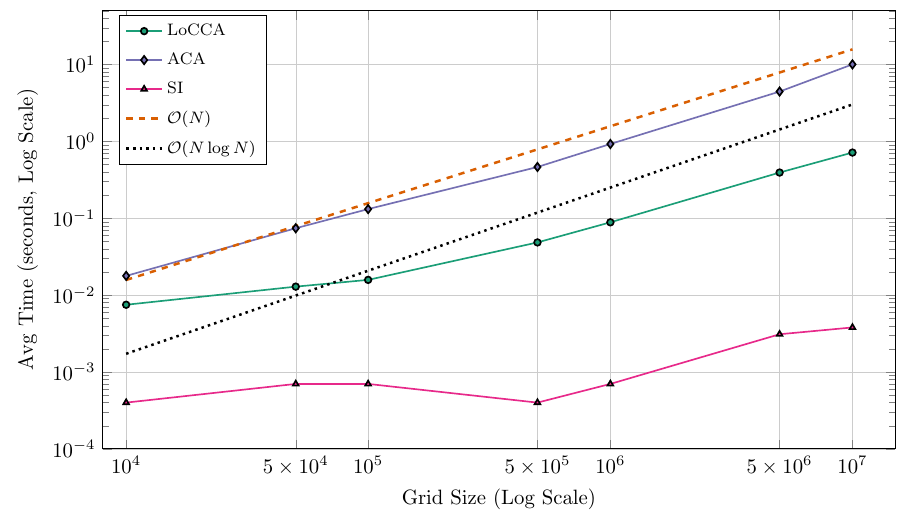}
        \caption{Execution time vs grid size $N$ ($\varepsilon = 10^{-8}$).}
        \label{fig: time vs grid}
    \end{subfigure}
    \begin{subfigure}[t]{0.48\linewidth}
        \centering
        \includegraphics[width=0.95\linewidth]{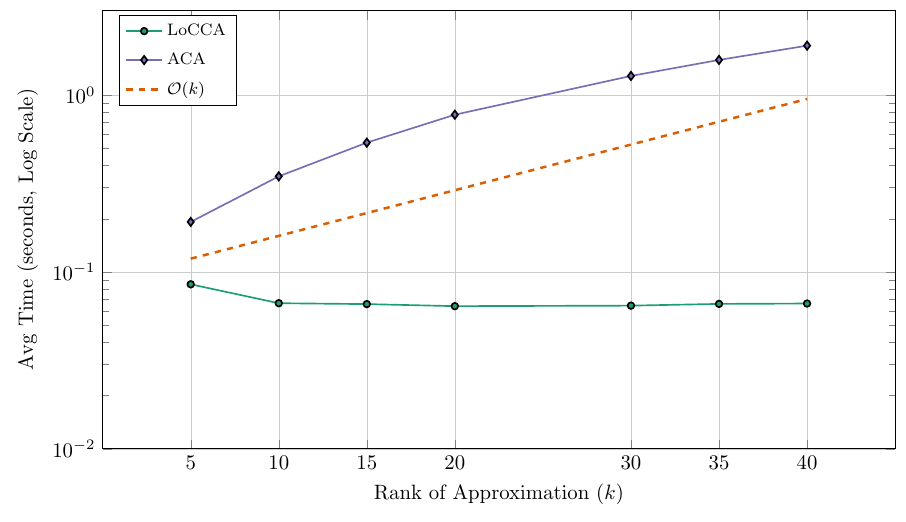}
        \caption{Execution time vs rank $r$ ($N = 10^6, \varepsilon = 10^{-8}$).}
        \label{fig: time vs rank}
    \end{subfigure}
    \begin{subfigure}[t]{0.48\linewidth}
        \centering
        \includegraphics[width=0.95\linewidth]{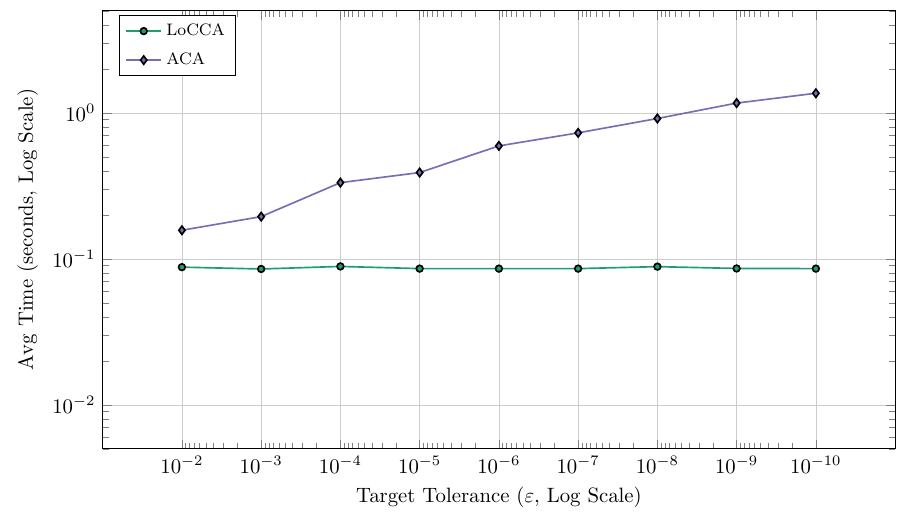}
        \caption{Execution time vs tolerance $\varepsilon$ ($N = 10^6$).}
        \label{fig: time vs tol}
    \end{subfigure}
    \caption{Execution time comparisons  for the $1/r$ kernel over $\mclx = [-3, -1]\times [0,2]$ and $\mcly = [1, 3] \times[0,2]$ ($p = 7, P = 49, l = 1$), averaged over 10 independent trials. (a) Execution time vs. grid size $N$. (b–c) Execution time vs. rank $r$ and tolerance $\varepsilon$ for $N = 10^6$.}
    \label{fig: time experiments}
\end{figure}

To evaluate the geometric robustness of the proposed framework, we benchmark LoCCA on challenging non-convex and non-standard domain configurations where classical grid-bound methods encounter fundamental limitations. Specifically, we study two complex geometric configurations given below.
\begin{enumerate}
    \item \textbf{Crescent-and-Core Domains:} Disconnected, non-convex target ($\mclx$) and source ($\mcly$) geometries consisting of outer crescent arcs coupled with asymmetric inner core disks (\autoref{fig: arc_dot_domains_data}).
    \item \textbf{Concentric Ring-and-Core Domains:} A central circular target disk $\mclx$ fully enclosed by an outer annular source ring $\mcly$ (\autoref{fig: concentric_domains_data}).
\end{enumerate}

\begin{figure}[H]
    \centering
    \begin{subfigure}[t]{0.3\linewidth}
        \includegraphics[width=0.87\linewidth]{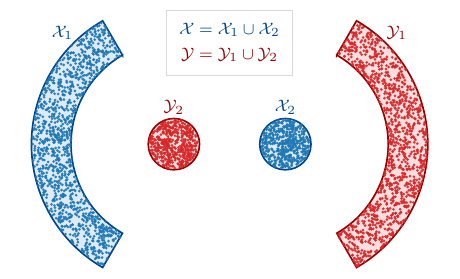}
        \caption{The Crescent-and-Core Domain layout}
        \label{fig: arc dot domains}
    \end{subfigure}\\
    \begin{subfigure}[t]{0.48\linewidth}
        \centering
        \includegraphics[width=0.87\linewidth]{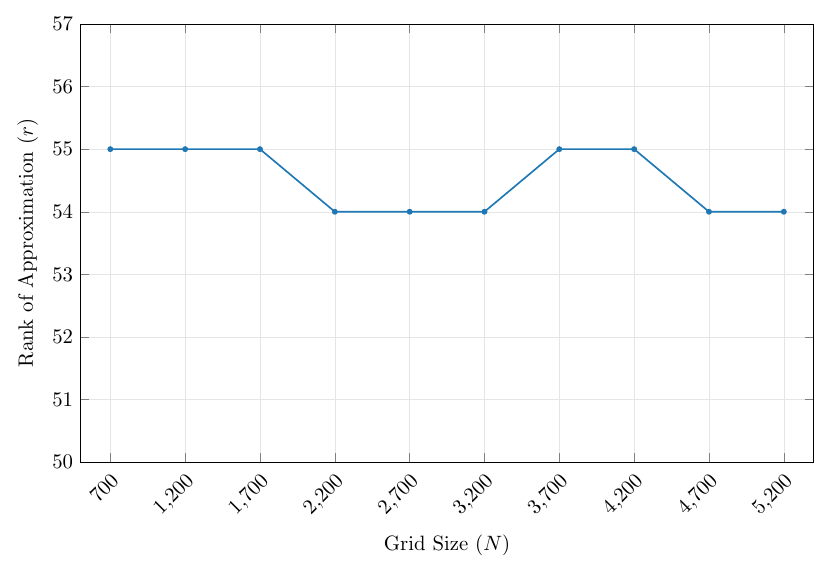}
        \caption{Grid size(N) vs rank of approximated matrix}
        \label{fig: arc dot grid vs rank}
    \end{subfigure}
    \begin{subfigure}[t]{0.48\linewidth}
        \centering
        \includegraphics[width=\linewidth]{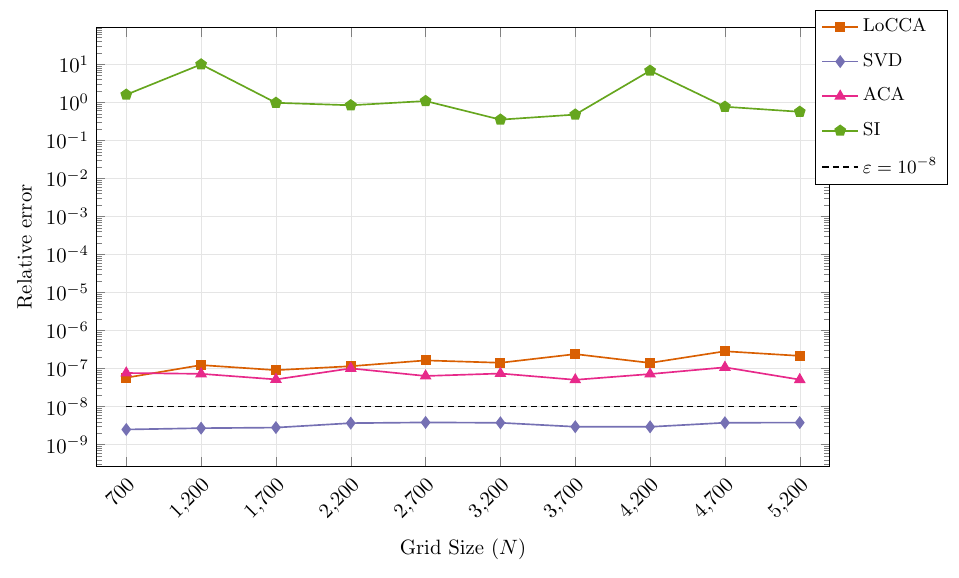}
        \caption{Grid size(N) vs relative error}
        \label{fig: arc dot grid vs relative error}
    \end{subfigure}
    \caption{Numerical benchmarking on disconnected, non-convex Crescent-and-Core geometries for the kernel $1/r$ ($ P = 289$ candidate nodes, $l = 1$, target tolerance $\varepsilon = 10^{-8}$). (a) Spatial layout of the disconnected target ($\mclx$, blue) and source ($\mcly$, red) point clouds. (b) Approximated numerical rank $r$ as a function of total grid size $N$. (c) Relative error for different grid size $N$. 
    }
    \label{fig: arc_dot_domains_data}
\end{figure}

\begin{figure}[ht]
    \centering
    \begin{subfigure}[t]{0.22\linewidth}
        \includegraphics[width=0.87\linewidth]{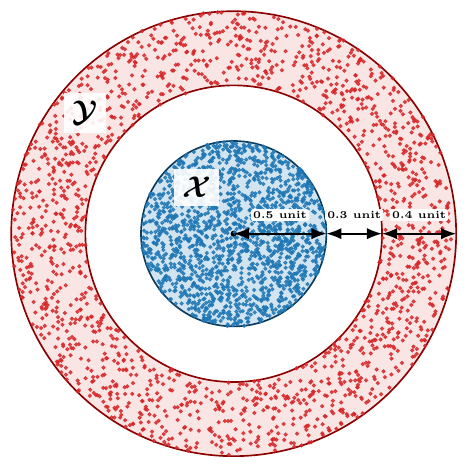}
        \caption{Concentric Ring-and-Core domain}
        \label{fig: concentric domains}
    \end{subfigure}\\
    \begin{subfigure}[t]{0.48\linewidth}
        \centering
        \includegraphics[width=0.87\linewidth]{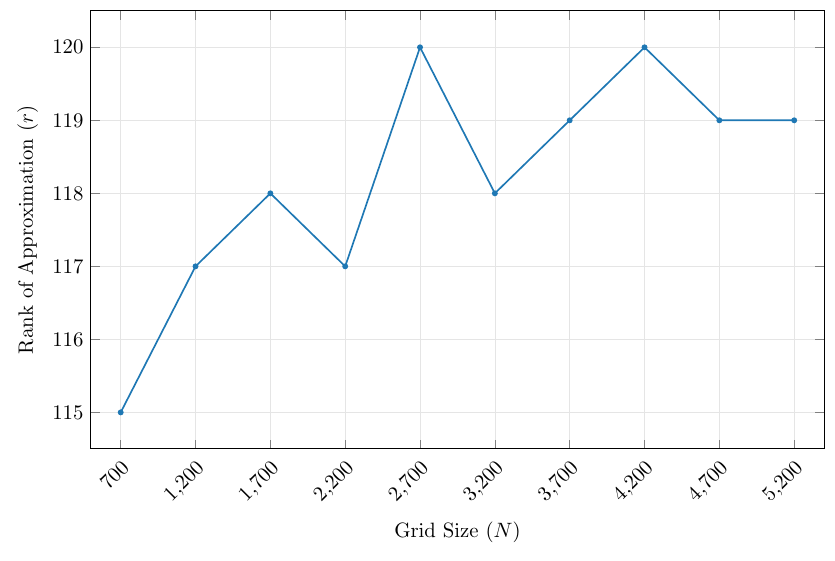}
        \caption{Grid size(N) vs rank of approximated matrix}
        \label{fig: grid vs rank}
    \end{subfigure}
    \begin{subfigure}[t]{0.48\linewidth}
        \centering
        \includegraphics[width=\linewidth]{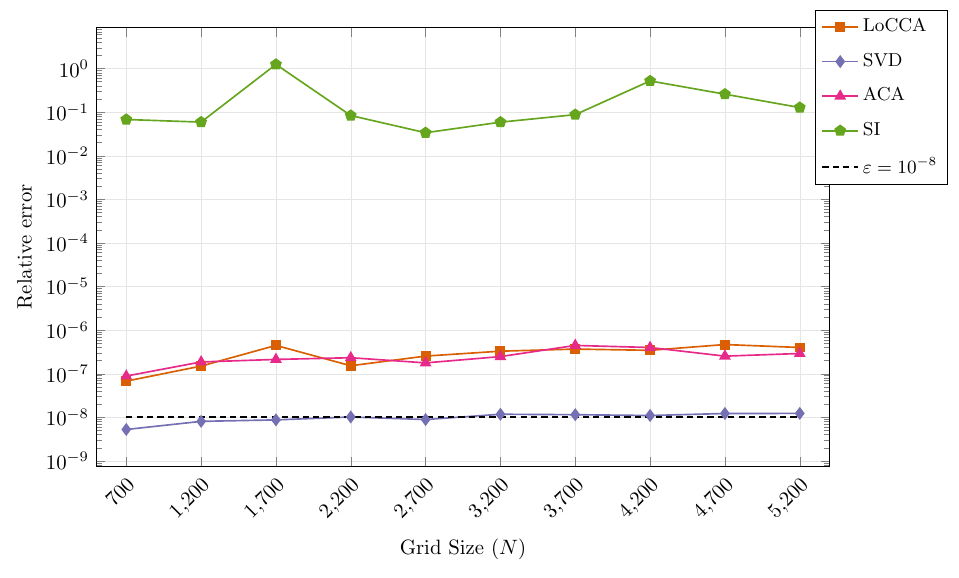}
        \caption{Grid size(N) vs relative error}
        \label{fig: grid vs relative error}
    \end{subfigure}
    \caption{Numerical benchmarking on Concentric Ring-and-Core geometries ($1/r$ kernel, $P = 289$, $l = 1$, target tolerance $\varepsilon = 10^{-8}$). Target domain $\mclx$ is a central disk ($r = 0.5$) enclosed by a concentric annular source ring $\mcly$ ($r_{\text{in}} = 0.8, r_{\text{out}} = 1.2$). (a) Spatial domain configuration. (b) Approximated rank scaling across expanding grid sizes $N$. (c) Relative error comparison.
    }
    \label{fig: concentric_domains_data}
\end{figure}

These non-standard geometries illustrate a primary failure mode of Skeletonized Interpolation (SI). Because SI relies on structured tensor-product Chebyshev grids defined over bounding hypercubes, its interpolation nodes inevitably sample empty, unphysical space or regions where the kernel is nearly singular. Consequently, as shown in \autoref{fig: arc dot grid vs relative error} and \autoref{fig: grid vs relative error}, SI fails to achieve accurate approximations on these configurations, exhibiting relative errors near $\mathcal{O}(1)$. 

In contrast, LoCCA bridges continuous reference grids and unstructured point clouds by mapping candidate anchors strictly to the physically nearest neighbors within the domain. Across both complex benchmark geometries, LoCCA maintains stable error control as grid resolution $N$ increases.

\section{Conclusion}\label{sec: conclusion}

In this paper, we introduced the Localized Chebyshev Cross Approximation (LoCCA) method, a data-driven and geometry-aware low-rank approximation technique for kernel matrices. By mapping continuous tensor-product Chebyshev nodes onto discrete physical nearest neighbors via LoC-NS, LoCCA bridges continuous function approximation theory and discrete algebraic skeletonization without requiring rigid tensor-product grid structures.

On the theoretical front, we established a comprehensive perturbation theory for Chebyshev polynomial interpolation under localized coordinate displacements. We proved an unconditional, a priori upper bound on the perturbed Lebesgue constant (\autoref{thm: unconditional bound}), demonstrating that localized spatial displacements preserve operator conditioning without catastrophic norm growth. We established sharp deterministic bounds on the deviation between unperturbed and perturbed Lagrange interpolants (\autoref{thm: Stability under node perturbations}). Specifically, we proved that $\Vert{}I_p f - \tilde{I}_p f\Vert{}_\linf$ remains uniformly stable, scaling linearly with the maximum node displacement $\delta$ with a growth constant depending on $p$ for any $f \in C^1[-1,1]$.
Leveraging directional tensor-product and telescoping expansions, we extended these stability bounds to $d$-dimensional continuous hypercubes $\Omega = [-1,1]^d$. Under the assumption of Strong RRQR factorization, we presented a global cross-approximation error convergence result (\autoref{thm: LoCCA convergence theorem}) that guarantees deterministic error control down to a prescribed tolerance $\varepsilon$ on physical point clouds.

Complementing the analytical results, our framework, comprising the Localized Chebyshev Node Selection (LoC-NS) and the cross-approximation, demonstrates computational efficiency and geometric resilience across numerical benchmarks. In practice, LoCCA isolates low-rank matrix skeletons whose ranks closely mirror the minimal theoretical rank $r_{\mathrm{SVD}}$ prescribed by optimal truncated SVD. LoCCA achieves approximately a $10\times$ speedup over ACA, and execution time remains remarkably flat and invariant with respect to target rank $r$ and precision tolerance $\varepsilon$. In non-convex and disconnected geometries, SI experiences catastrophic breakdown (a $\mathcal{O}(1)$ relative error); by contrast, LoCCA maintains strict error control within the given tolerance.

While \autoref{thm: unconditional bound} proves $\tilde{\Lambda}_p \le p^C \Lambda_p$, numerical experiments suggest that the polynomial factor $p^C$ is a surplus estimate. Formally proving \autoref{rmk: sharp_lebesgue_conjecture} to establish the sharp logarithmic asymptotic growth $\tilde{\Lambda}_p = \mathcal{O}(\log p + \delta)$ remains an interesting open problem.

\appendix

\section{Some Auxiliary Results used in the Article}

\begin{mylma}\label{lma: weight-ratio-bound}
Let $\omega_{\mathbf{t}}(x) = \prod_{j=1}^p (x - t_j)$ denote the node-generating polynomial associated with a node vector $\mathbf{t} \in [-1,1]^p$, and let $\boldsymbol{\xi} = [\xi_1, \dots, \xi_p]^T$ be the Chebyshev nodes defined by \eqref{equ: cheb nodes}. Suppose the intermediate node configuration $\boldsymbol{\eta} = [\eta_1, \dots, \eta_p]^T \in [-1,1]^p$ satisfies $|\eta_j - \xi_j| < \delta$ for all $j = 1, \dots, p$. Then, under the uniform threshold constraint \eqref{equ: delta strict threshold} $\delta \le m/4$, the ratio of the polynomial derivative weights of $\omega_{\boldsymbol{\eta}}(x)$ satisfies 
\[
    \left| \frac{\omega'_{\boldsymbol{\eta}}(\eta_k)}{\omega'_{\boldsymbol{\eta}}(\eta_i)} \right| \le C_0 p^\alpha, \qquad \text{for all } i,k = 1, \dots, p,
\]
where $C_0$ and $\alpha$ are uniform constants independent of $p$, $i$, and $k$.
\end{mylma}

\begin{proof}
Let $\omega_{\mathbf{t}}(x) = \prod_{j=1}^p (x - t_j)$ denote the node-generating polynomial associated with a parameter $\mathbf{t} \in [-1,1]^p$. Under the intermediate state $\mathbf{t} = \boldsymbol{\eta}$, we define $\omega_{\boldsymbol{\eta}}(x) = \prod_{j=1}^p (x - \eta_j)$, whereas $\omega_{\boldsymbol{\xi}}(x) = \prod_{j=1}^p (x - \xi_j)$ represents the Chebyshev node-generating polynomial. Now, since $|\eta_j - \xi_j| < \delta$ where $\delta$ satisfies \eqref{equ: delta strict threshold}, from \eqref{equ: eta node separation}, we can ensure that the elements of $\boldsymbol{\eta}$ remain strictly distinct, yielding the derivative weight $|\omega'_{\boldsymbol{\eta}}(\eta_k)| = \prod_{j \neq k} |\eta_k - \eta_j|$. Now, the derivative weights satisfy
\begin{equation}\label{equ: log of wiight coefficient ratio}
    \log \frac{|\omega_{\boldsymbol{\xi}}'(\xi_k)|}{|\omega'_{\boldsymbol{\eta}}(\eta_k)|} = \sum_{\substack{j=1 \\ j \neq k}}^p \log \frac{|\xi_k - \xi_j|}{|\eta_k - \eta_j|}.
\end{equation}
Using separation margin $|\eta_k - \eta_j| \ge \frac{1}{2}|\xi_k - \xi_j|$ and applying the triangle inequality yields $\frac{|\xi_k - \xi_j|}{|\eta_k - \eta_j|} \le 1 + \frac{|\eta_k - \xi_k| + |\eta_j - \xi_j|}{|\eta_k - \eta_j|}$. Now as $\log(1+t) \le t$, then using the proximity condition $|\eta_m - \xi_m| < \delta$, we have 
\[
\log \frac{|\xi_k - \xi_j|}{|\eta_k - \eta_j|} \le \frac{|\eta_k - \xi_k| + |\eta_j - \xi_j|}{|\eta_k - \eta_j|} \le \frac{2\delta}{\frac{1}{2}|\xi_k - \xi_j|} = \frac{4\delta}{|\xi_k - \xi_j|}.
\]
Now, summing the above for $j = 1 : p$ with $j\neq k$ for some $k$ and then using \autoref{lma: chebyshev-harmonic} and the spacing constraint $\delta p^2 \le C_2$ from \eqref{equ: log of wiight coefficient ratio} we have 
\[
    \log \frac{|\omega_{\boldsymbol{\xi}}'(\xi_k)|}{|\omega'_{\boldsymbol{\eta}}(\eta_k)|} \le 4\delta \sum_{\substack{j=1 \\ j \neq k}}^p \frac{1}{|\xi_k - \xi_j|} \le 4 C_H \delta p^2 \log p \le 4 C_H C_2 \log p \implies \frac{|\omega_{\boldsymbol{\xi}}'(\xi_k)|}{|\omega'_{\boldsymbol{\eta}}(\eta_k)|} \le p^{4 C_H C_2} = p^{C_k}
\]

Similarly, for the reciprocal quotient, we can show $\frac{|\omega_{\boldsymbol{\xi}}'(\xi_i)|}{|\omega'_{\boldsymbol{\eta}}(\eta_i)|} \le p^{C_i}$ for a uniform constant $C_i > 0$. Now, the required weight ratio can be written as follows
\begin{equation}\label{equ: single_weight_poly_bound}
    \left| \frac{\omega'_{\boldsymbol{\eta}}(\eta_k)}{\omega'_{\boldsymbol{\eta}}(\eta_i)} \right| =  \frac{|\omega'_{\boldsymbol{\eta}}(\eta_k)|}{|\omega_{\boldsymbol{\xi}}'(\xi_k)|} \cdot \frac{|\omega_{\boldsymbol{\xi}}'(\xi_i)|}{|\omega'_{\boldsymbol{\eta}}(\eta_i)|} \cdot \left| \frac{\omega_{\boldsymbol{\xi}}'(\xi_k)}{\omega_{\boldsymbol{\xi}}'(\xi_i)} \right|.
\end{equation}
As we already know, the first two components of \eqref{equ: single_weight_poly_bound} scale as $p^{\mclo{1}}$ and \eqref{equ: cheb weight ratio bound}, establishing a stable uniform polynomial bound $C_0 p^\alpha$ with $\alpha = C_k +C_i + 1 > 0$.
\end{proof}

\begin{myremark}\label{rmk: cheb derivatives}
The monic node-generating polynomial associated with the Chebyshev roots satisfies $\omega_{\boldsymbol{\xi}}(x) = 2^{1-p} T_p(x)$. Differentiating using the trigonometric mapping $x = \cos\theta$ and $T_p(\cos\theta) = \cos(p\theta)$ yields
\begin{equation}
    T_p'(x) = \frac{\mathrm{d}}{\mathrm{d}\theta}[\cos(p\theta)] \frac{\mathrm{d}\theta}{\mathrm{d}x} = \frac{-p\sin(p\theta)}{-\sin\theta} = p \frac{\sin(p\theta)}{\sin\theta}.
\end{equation}
Now at $\xi_m = \cos \theta_m$ where $p\theta_m = (2m-1)\frac{\pi}{2}$ implies $|\sin(p\theta_m)| = 1$. It follows that
\begin{equation}
    |\omega_p'(\xi_m)| = \frac{p}{2^{p-1}\sqrt{1 - \xi_m^2}} \implies  \frac{|\omega_p'(\xi_k)|}{|\omega_p'(\xi_i)|} = \frac{\sqrt{1-\xi_i^2}}{\sqrt{1-\xi_k^2}}.
\end{equation}
As the denominator attains its minimum at $k=1$ or $k=p$, so at $\theta_1 = \frac{\pi}{2p}$, using Jordan's inequality we have $\sqrt{1-\xi_1^2} = \sin(\pi/2p) \ge {1}/{p}.$ 
Since $\sqrt{1-\xi_i^2} \le 1$ globally on $[-1,1]$, we conclude 
\begin{equation} \label{equ: cheb weight ratio bound}
    \frac{|\omega_p'(\xi_k)|}{|\omega_p'(\xi_i)|} \le \frac{1}{1/p} = p.
\end{equation}
\end{myremark}

\section{Numerical Investigation of Perturbed Chebyshev Interpolation Stability on Runge's Function}

To empirically validate the operator stability bounds proved in \autoref{sec: Theoretical Analysis}, we examine the performance of perturbed Chebyshev interpolation on the Runge function.
\begin{equation}
    f(x) = \frac{1}{1 + 25x^2}, \quad x \in [-1, 1].
    \label{equ:runge_function}
\end{equation}

The primary objective of this is to check whether displacing Chebyshev nodes by the admissible perturbation $\delta \le {m}/{4}$ preserves operator stability and retains the optimal asymptotic convergence rate as established by \autoref{thm: unconditional bound} and \autoref{thm: Stability under node perturbations}.

\begin{figure}[htbp]
    \centering
    \includegraphics[width=\textwidth]{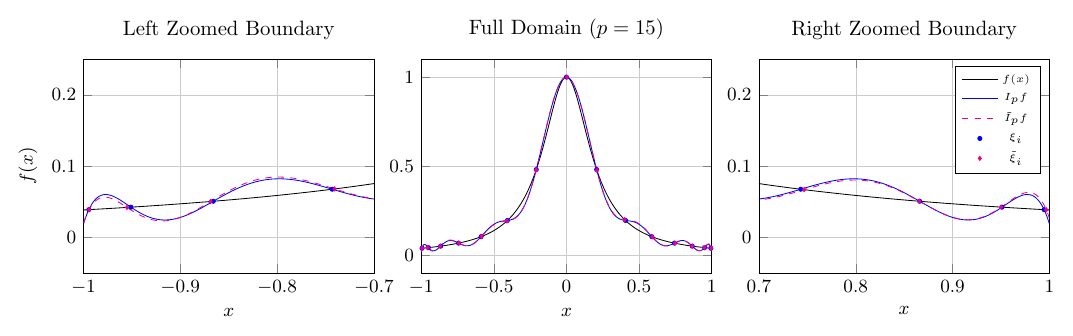}
    \caption{Interpolation of $f(x)$ with $p = 15$ nodes. The center panel displays the full domain $x \in [-1, 1]$, while the left and right panels detail boundary behavior at $x \in [-1.0, -0.7]$ and $x \in [0.7, 1.0]$, respectively. 
    }
    \label{fig: runge_interpolation}
\end{figure}

As illustrated in \autoref{fig: runge_interpolation}, polynomial interpolation on perturbed Chebyshev nodes remains well-conditioned across the entire domain $[-1, 1]$. In particular, in the zoomed boundary panels, we can see that the distance between the unperturbed interpolant $I_p f$ and the perturbed interpolant $\tilde{I}_p f$ remains tightly controlled by the perturbation $\delta$, which provides visual confirmation of \autoref{thm: Stability under node perturbations}. 

\begin{figure}[htbp]
    \centering
    \begin{subfigure}[t]{0.49\linewidth}
        \includegraphics[width=0.95\linewidth]{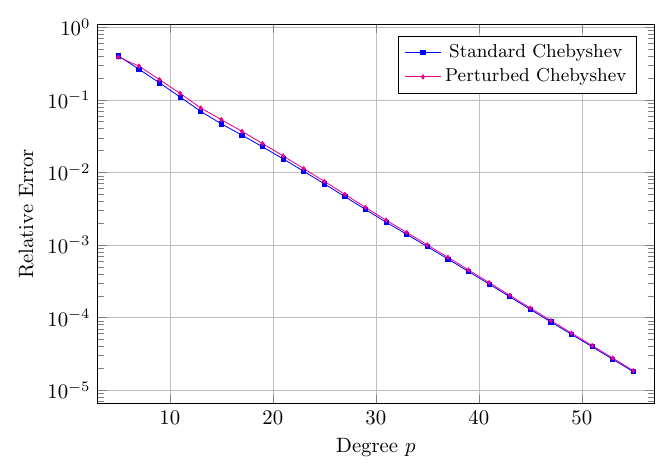}
        \caption{Decay of Chebyshev and Perturbed interpolants.}
        \label{fig: runge_decay}
    \end{subfigure}
    \begin{subfigure}[t]{0.49\linewidth}
        \includegraphics[width=0.95\linewidth]{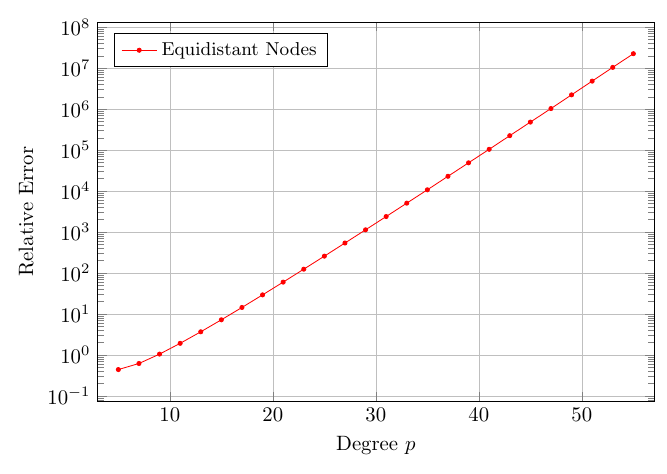}
        \caption{Divergence due to equidistant nodes.}
        \label{fig: runge_divergence_comparison}
    \end{subfigure}
    \caption{Relative error convergence due to interpolation for $p \in [5, 55]$ number of nodes.}
    \label{fig: runge_convergence_study}
\end{figure}

\autoref{fig: runge_convergence_study} provides quantitative verification of theoretical estimates. As depicted in \autoref{fig: runge_divergence_comparison}, interpolation due to equidistant nodes suffers from exponential error explosion due to the growth of the Lebesgue constant $\Lambda_p^{\text{eq}} $, which is of $\mclo{{2^p}/{e p \log p}}$ \cite{hesthaven1998electrostatics}. \autoref{fig: runge_decay} shows that the perturbed interpolant $\tilde{I}_p f$ closely tracks the unperturbed Chebyshev interpolant $I_p f$. This provides direct empirical validation of \autoref{thm: unconditional bound}, verifying that the perturbed Lebesgue constant remains polynomial-logarithmically bounded.  
    


\section{Tabular Benchmarks on Vertex-Sharing Domains} 

In this section, we present \autoref{tab: locca_performance_time_tolerance_comparison} showing relative errors and average execution times for kernels listed in \eqref{equ: kernel list} for \textit{vertex-sharing domain setup} $\mclx = [-1,1]\times[0,2] $ and $ \mcly = [1,3]\times[2,4],$ where the target and source domains share the vertex at $(1,2)$. 

\begin{table}[ht]
\centering
\resizebox{0.8\textwidth}{!}{%
\begin{tabular}{llcccccc}
\toprule
& & \multicolumn{3}{c}{\textbf{Relative Error} \textit{(Uniform Grid)}} & \multicolumn{3}{c}{\textbf{Avg. Time} \textit{(seconds, 10 trials)}} \\
\cmidrule(lr){3-5} \cmidrule(lr){6-8}
\textbf{Kernel} & \textbf{Method}
& $\varepsilon=10^{-4}$
& $\varepsilon=10^{-6}$
& $\varepsilon=10^{-8}$
& $\varepsilon=10^{-4}$
& $\varepsilon=10^{-6}$
& $\varepsilon=10^{-8}$ \\
\midrule

\multirow{4}{*}{$1/r$}
& Rank $(r)$ &25 &46 &73 & -- & -- & -- \\
& LoCCA      &$7.926707\times 10^{-4}$ &$1.796203\times 10^{-5}$ &$5.836513\times 10^{-7}$ & $1.006677$ & $0.993188$ & $1.065765$ \\
& ACA        &$7.610625\times 10^{-3}$ &$1.496272\times 10^{-4}$ &$1.589126\times 10^{-5}$ & $0.432318$ & $1.038314$ & $1.915775$ \\
& SI         &$2.784097\times 10^{-4}$ &$3.899703\times 10^{-6}$ &$2.869157\times 10^{-7}$ & $0.941954$ & $0.930513$ & $0.986161$ \\
\midrule

\multirow{4}{*}{$\log r$}
& Rank $(r)$ &20 &31 &42 & -- & -- & -- \\
& LoCCA      &$2.403002\times 10^{-5}$ &$1.092055\times 10^{-7}$ &$2.259749\times 10^{-9}$ & $0.985694$ & $1.003111$ & $0.934361$ \\
& ACA        &$7.862198\times 10^{-4}$ &$2.934027\times 10^{-6}$ &$4.321152\times 10^{-8}$ & $0.470508$ & $0.757769$ & $1.119425$ \\
& SI         &$4.139327\times 10^{-5}$ &$4.043177\times 10^{-8}$ &$1.540308\times 10^{-9}$ & $0.897036$ & $0.945746$ & $0.924446$ \\
\midrule

\multirow{4}{*}{$\dfrac{\cos r}{r}$}
& Rank $(r)$ &29 &50 &74 & -- & -- & -- \\
& LoCCA      &$2.772655\times 10^{-3}$ &$2.348328\times 10^{-5}$ &$7.945583\times 10^{-7}$ & $1.109902$ & $1.097676$ & $1.105714$ \\
& ACA        &$3.841560\times 10^{-3}$ &$3.272173\times 10^{-4}$ &$1.109720\times 10^{-5}$ & $0.931978$ & $1.849091$ & $3.042118$ \\
& SI         &$1.918470\times 10^{-3}$ &$7.408377\times 10^{-6}$ &$4.339765\times 10^{-7}$ & $1.040185$ & $0.999815$ & $1.037307$ \\
\midrule

\multirow{4}{*}{$e^{-r^2}$}
& Rank $(r)$ &16 &29 &45 & -- & -- & -- \\
& LoCCA      &$4.024449\times 10^{-4}$ &$6.609316\times 10^{-6}$ &$4.224703\times 10^{-8}$ & $1.196488$ & $1.164910$ & $1.146536$ \\
& ACA        &$1.970275\times 10^{-3}$ &$3.371133\times 10^{-6}$ &$9.218742\times 10^{-8}$ & $0.588015$ & $1.073603$ & $1.921577$ \\
& SI         &$1.589025\times 10^{-4}$ &$7.548823\times 10^{-6}$ &$5.764549\times 10^{-8}$ & $1.102922$ & $1.126818$ & $1.102062$ \\
\midrule

\multirow{4}{*}{$\sqrt{1+r^2}$}
& Rank $(r)$ &10 &19 &31 & -- & -- & -- \\
& LoCCA      &$5.865750\times 10^{-4}$ &$7.470466\times 10^{-6}$ &$1.768371\times 10^{-8}$ & $1.015719$ & $1.017417$ & $1.025954$ \\
& ACA        &$2.215989\times 10^{-4}$ &$1.150627\times 10^{-6}$ &$2.585513\times 10^{-8}$ & $0.372572$ & $0.715170$ & $1.196391$ \\
& SI         &$1.185676\times 10^{-4}$ &$1.415365\times 10^{-7}$ &$4.444030\times 10^{-9}$ & $0.921868$ & $0.939795$ & $0.970981$ \\
\bottomrule
\end{tabular}
}

\caption{Comparison of relative errors and execution times (averaged over 10 trials) on the vertex-sharing domain setup $\mclx = [-1,1]\times[0,2]$ and $\mcly = [1,3]\times[2,4]$ touching at corner $(1,2)$, discretized at $N = 707^2 \approx 5 \times 10^5$ points per domain.  The rank $r$ determined by LoCCA is prescribed to ACA and SI for direct comparative analysis. LoCCA $p = 45$ ($P = 2025$ nodes) and $l = 1$. Here $p$ is chosen for $\varepsilon= 10^{-8}$, a better choice of $p$ for $\varepsilon= 10^{-4}$ and $10^{-6}$ can give better time of execution of LoCCA than ACA.}

\label{tab: locca_performance_time_tolerance_comparison}
\end{table}

\section{Spectral Decomposition using Localized Chebyshev Neighbors}

While LoCCA provides a highly efficient cross-approximation, the localized sampling framework presented in \autoref{alg: loc_nbds} is adaptable to other algebraic paradigms. Certain scientific computing applications (such as PCA, reduced-order modeling, and spectral filtering) require strictly orthogonal basis vectors alongside true singular values. To address this, we present a secondary structural utility of our framework: the \textit{Localized Chebyshev-Accelerated Truncated Singular Value Decomposition} (LoC-TSVD). \autoref{alg: LoC-TSVD} utilizing $\mathcal{N}_{\mclx}$ and $\mathcal{N}_{\mcly}$, constructs a compressed sub-matrix representation, and subsequently extracts optimal, orthonormal rank-$r$ spectral factors via economy-sized QR factorizations and SVD, requiring $\mathcal{O}(P^2N)$ time and $\mathcal{O}(PN)$ space complexity.

\begin{algorithm}[ht]
\caption{Localized Chebyshev-Accelerated Truncated SVD (LoC-TSVD)}
\label{alg: LoC-TSVD}
\begin{algorithmic}[1]
\State \textbf{Input:} Grids $X \subset \mclx$, $Y \subset \mcly$, kernel $\mclk$, $P$, $l$, tolerance $\varepsilon$
\State \textbf{Output:} $U_r \in \mathbb{R}^{n\times r}$, $\Sigma_r \in \mathbb{R}^{r\times r}$, $V_r \in \mathbb{R}^{m\times r}$ and effective rank $r$

\State $[\mathcal{N}_\mclx, \mathcal{N}_\mcly] \leftarrow \text{\textbf{Call \autoref{alg: loc_nbds}}}(X, Y, P, l)$ 

\State Construct: 
\[
    \mathbf{K}_x = \mathcal{K}(X, \mathcal{N}_\mcly), \quad 
    \mathbf{K}   = \mathcal{K}(\mathcal{N}_\mclx, \mathcal{N}_\mcly), \quad 
    \mathbf{K}_y = \mathcal{K}(\mathcal{N}_\mclx, Y)
\]

\State Compute thin QR factorizations: 
\[ \mathbf{K}_x = Q_1 R_1, \qquad \mathbf{K}_y^T = Q_2 R_2 \]

\State Compute the small core matrix: 
\[ \widehat{\mathbf{K}} = R_1 \mathbf{K}^{-1}R_2^{T} \] 

\State Compute the SVD: 
\[ \widehat{\mathbf{K}} = \widetilde{U}\Sigma \widetilde{V}^T \]

\State Determine effective rank: $r = \#\{ \sigma_i : \sigma_i \ge \varepsilon \sigma_1\}$ 

\State Retain the leading $r$ components $\widetilde{U}_r, \Sigma_r, \widetilde{V}_r$ and compute: 
\[ U_r = Q_1 \widetilde{U}_r, \quad V_r = Q_2 \widetilde{V}_r \]

\State \Return $\mathcal{K}(X,Y) \approx  U_r \, \Sigma_r\, V_r^T$
\end{algorithmic}
\end{algorithm}

To empirically validate the LoC-TSVD (\autoref{alg: LoC-TSVD}), we compare its numerical performance against the exact truncated SVD. As summarized in \autoref{tab: LoC-TSVD relative error data table with tolerance comparison}, LoC-TSVD achieves relative errors virtually indistinguishable from exact truncated SVD across all kernels listed in \eqref{equ: kernel list} on both uniform and random domain discretizations.

Furthermore, \autoref{fig: boxplot_all_kernels_all_tols_TSVD} illustrates the error distribution of LoC-TSVD and exact SVD over $500$ independent random point cloud realizations. Across all tested tolerances $\varepsilon \in \{10^{-4}, 10^{-6}, 10^{-8}\}$, LoC-TSVD exhibits an extremely compact interquartile range that tightly mirrors the optimal lower bound of exact SVD truncation. These results confirm that LoC-TSVD extracts near-optimal, orthonormal spectral bases without computing dense $O(N^3)$ SVDs.

\begin{table}[ht]
\centering
\resizebox{0.9\textwidth}{!}{%
\begin{tabular}{llcccccc}
\toprule
& & \multicolumn{3}{c}{Uniform Grid} 
& \multicolumn{3}{c}{Random Grid (Mean over 500 trials)} \\

\cmidrule(lr){3-5} \cmidrule(lr){6-8}

Kernel & Method
& $\varepsilon=10^{-4}$
& $\varepsilon=10^{-6}$
& $\varepsilon=10^{-8}$
& $\varepsilon=10^{-4}$
& $\varepsilon=10^{-6}$
& $\varepsilon=10^{-8}$ \\

\midrule

\multirow{3}{*}{$1/r$}
& Rank $(r)$ &6  &12  &19  & -- & -- & -- \\
& LoC-TSVD &$1.035351\times 10^{-4}$  &$1.146664\times 10^{-6}$  &$1.415955\times 10^{-8}$  &$8.527126\times 10^{-5}$  &$8.627567\times 10^{-7}$  &$9.805651\times 10^{-9}$  \\
& SVD   &$1.035351\times 10^{-4}$  &$1.146664\times 10^{-6}$  &$1.415955\times 10^{-8}$  &$8.527126\times 10^{-5}$  &$8.627566\times 10^{-7}$  &$9.805647\times 10^{-9}$  \\
\midrule

\multirow{3}{*}{$\log r$}
& Rank $(r)$ &5  &9  &13  & -- & -- & -- \\
& LoC-TSVD &$4.242332\times 10^{-5}$  &$2.134395\times 10^{-7}$  &$1.332048\times 10^{-9}$  &$3.591682\times 10^{-5}$  &$1.617050\times 10^{-7}$  &$8.840601\times 10^{-10}$  \\
& SVD   &$4.242332\times 10^{-5}$  &$2.134395\times 10^{-7}$  &$1.332048\times 10^{-9}$  &$3.591682\times 10^{-5}$  &$1.617050\times 10^{-7}$  &$8.839129\times 10^{-10}$  \\
\midrule

\multirow{3}{*}{$\dfrac{\cos r}{r}$}
& Rank $(r)$ &10  &17  &25  & -- & -- & -- \\
& LoC-TSVD &$7.154636\times 10^{-5}$  &$9.996832\times 10^{-7}$  &$1.011671\times 10^{-8}$  &$5.698595\times 10^{-5}$  &$7.541357\times 10^{-7}$  &$6.936725\times 10^{-9}$  \\
& SVD   &$7.154636\times 10^{-5}$  &$9.996832\times 10^{-7}$  &$1.011671\times 10^{-8}$  &$5.698595\times 10^{-5}$  &$7.541357\times 10^{-7}$  &$6.935993\times 10^{-9}$  \\
\midrule

\multirow{3}{*}{$e^{-r^2}$}
& Rank $(r)$ &12  &23  &35  & -- & -- & -- \\
& LoC-TSVD &$8.905985\times 10^{-5}$  &$7.989985\times 10^{-7}$  &$1.084042\times 10^{-8}$  &$1.250953\times 10^{-4}$  &$2.465211\times 10^{-3}$  &$3.908079\times 10^{-4}$  \\
& SVD   &$8.905985\times 10^{-5}$  &$7.989982\times 10^{-7}$  &$1.081347\times 10^{-8}$  &$7.146567\times 10^{-5}$  &$5.662742\times 10^{-7}$  &$6.775901\times 10^{-9}$  \\
\midrule

\multirow{3}{*}{$\sqrt{1+r^2}$}
& Rank $(r)$ &6  &10  &15  & -- & -- & -- \\
& LoC-TSVD &$1.009817\times 10^{-5}$  &$3.520175\times 10^{-7}$  &$1.147103\times 10^{-8}$  &$8.568534\times 10^{-6}$  &$2.764857\times 10^{-7}$  &$8.606240\times 10^{-9}$  \\
& SVD   &$1.009817\times 10^{-5}$  &$3.520175\times 10^{-7}$  &$1.147103\times 10^{-8}$  &$8.568534\times 10^{-6}$  &$2.764857\times 10^{-7}$  &$8.606239\times 10^{-9}$  \\
\bottomrule
\end{tabular}
}
\caption{Comparison of relative errors between LoC-TSVD and exact truncated SVD for target tolerances $\varepsilon \in \{10^{-4},10^{-6},10^{-8}\}$ over $\mclx = [-3,-1]\times[0,2]$ and $\mcly = [1,3]\times[0,2]$ ($N = 65^2 = 4225$ points per domain, $P = 64$ candidate nodes and neighbor size $l = 2$). For random grids, reported metrics represent mean relative errors over 500 independent point realizations using the rank $r$ determined on the uniform grid.}

\label{tab: LoC-TSVD relative error data table with tolerance comparison}
\end{table}

\begin{figure}[ht]
    \centering
    \begin{subfigure}[t]{0.49\linewidth}
        \centering
        \includegraphics[width=\linewidth]{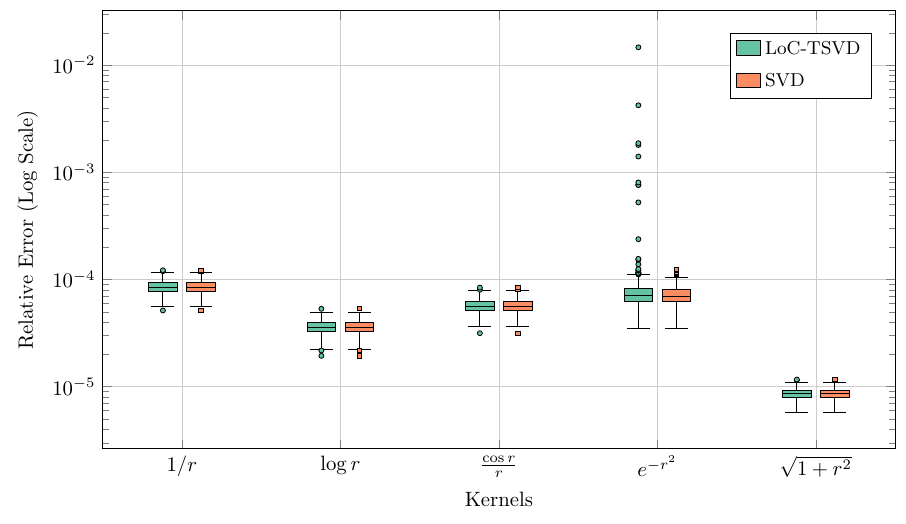}
        \caption{$\varepsilon = 1.0\textrm{e}\!-\!4$}
        \label{fig: box tol 1.0e-4}
    \end{subfigure}
    \begin{subfigure}[t]{0.49\linewidth}
        \centering
        \includegraphics[width=\linewidth]{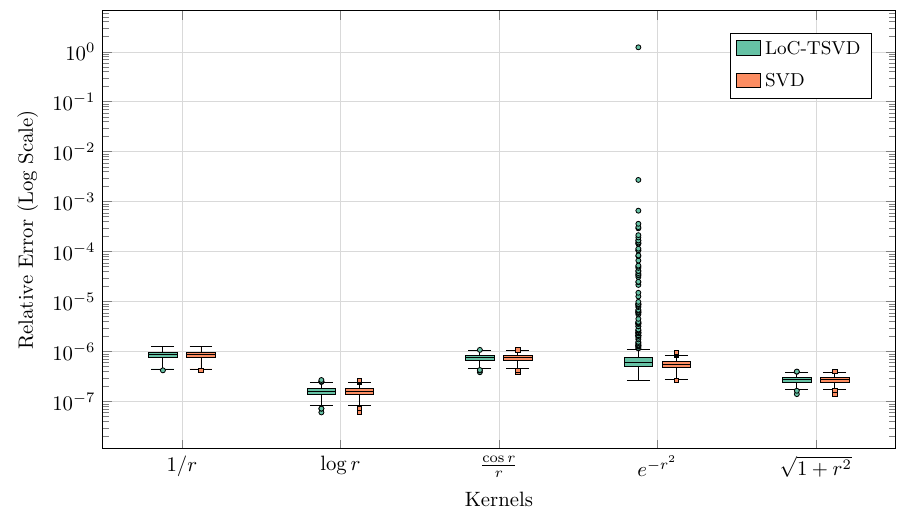}
        \caption{$\varepsilon = 1.0\textrm{e}\!-\!6$}
        \label{fig: box tol 1.0e-6}
    \end{subfigure}
    \begin{subfigure}[t]{0.49\linewidth}
        \centering
        \includegraphics[width=\linewidth]{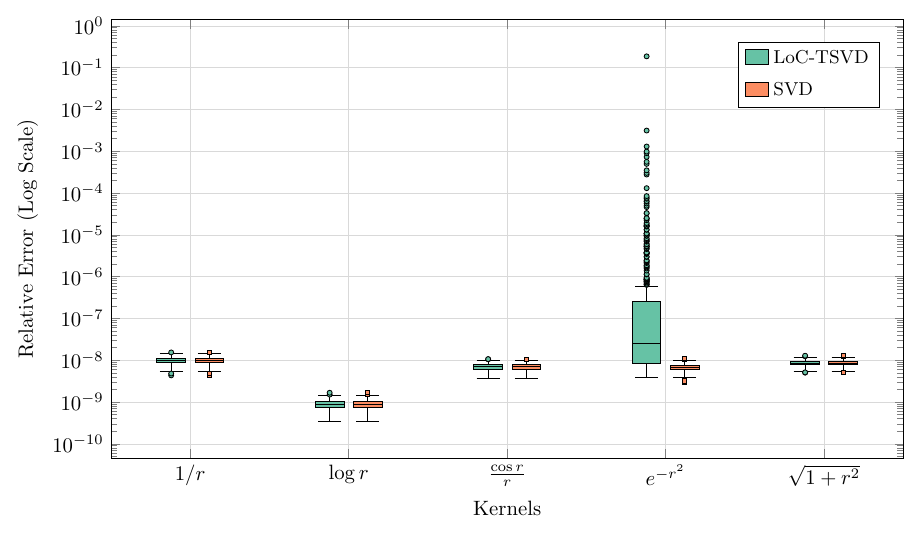}
        \caption{$\varepsilon = 1.0\textrm{e}\!-\!8$}
        \label{fig: box tol 1.0e-8}
    \end{subfigure}
    \caption{Relative errors computed over $500$ independent random discretizations of the domains $\mclx = [-3,-1]\times[0,2]$ and $\mcly = [1,3]\times[0,2]$ for all the kernels as mentioned in \eqref{equ: kernel list}. Given a tolerance $\varepsilon \in \{10^{-4},10^{-6},10^{-8}\}$ and for each method and kernel, the distribution of errors across random grids is summarized using boxplots: the central line indicates the median, the box represents the interquartile range (25th to 75th percentiles), whiskers indicate the range of typical (non-outlier) values, and individual points denote outliers.} 
    \label{fig: boxplot_all_kernels_all_tols_TSVD}
\end{figure}

To provide fine-grained spectral validation of the LoC-TSVD, we evaluate individual singular values against the exact analytical SVD spectrum. 
As shown in \autoref{fig: decay of singular values}, the normalized singular values computed by LoC-TSVD overlay perfectly with those of the exact SVD, reproducing the exact exponential decay. Furthermore, \autoref{fig: singular value differences} demonstrates that the element-wise normalized error $\Delta_i = |(\tilde{\sigma}_i / \tilde{\sigma}_1) - (\sigma_i / \sigma_1)|$ remains strictly bounded below the prescribed truncation threshold ($\varepsilon = 10^{-8}$). 

\begin{figure}[ht]
    \centering
    \begin{subfigure}[t]{0.48\linewidth}
        \centering
        \includegraphics[width=\linewidth]{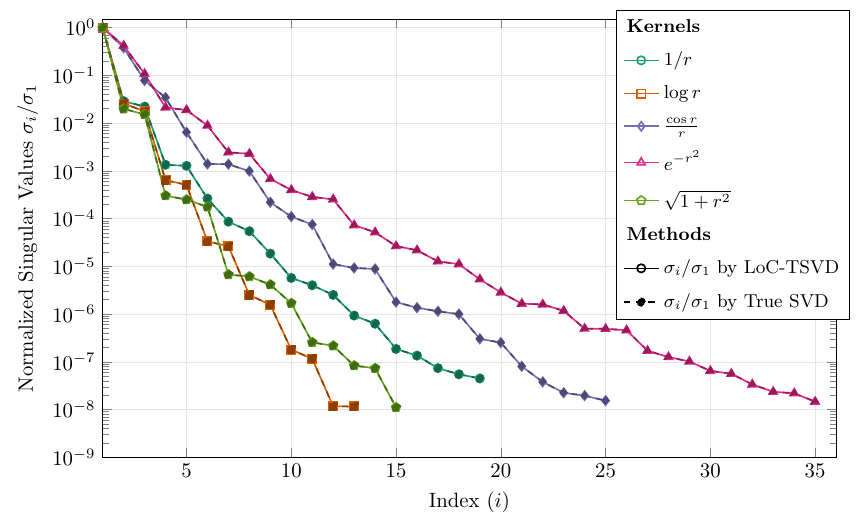}
        \caption{The singular values by LoC-TSVD and true SVD}
        \label{fig: decay of singular values}
    \end{subfigure}\hskip 0.4cm
    \begin{subfigure}[t]{0.48\linewidth}
        \centering
        \includegraphics[width=\linewidth]{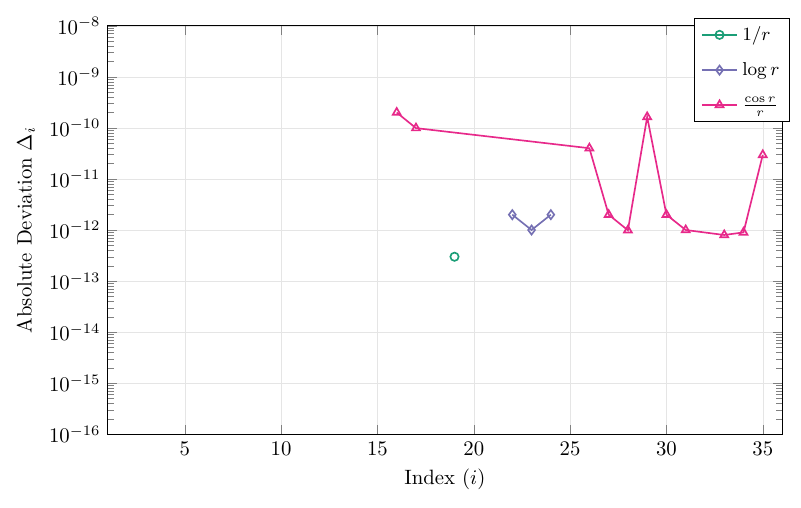}
        \caption{difference of normalized singular values of LoC-tsvd and actual SVD}
        \label{fig: singular value differences}
    \end{subfigure}
    \caption{Spectral decay profile and element-wise discrepancy validation for LoC-TSVD against exact SVD truncation across all kernel operators in \eqref{equ: kernel list} on a uniform grid at target tolerance $\varepsilon = 10^{-8}$. (a) Normalized singular value spectra ($\sigma_i / \sigma_1$) as a function of index $i$, demonstrating perfect overlay with exact SVD decay. (b) Absolute element-wise discrepancy $\Delta_i$ for evaluated singular value components, demonstrating that error deviations remain strictly bounded below the prescribed truncation floor $\varepsilon = 10^{-8}$.}
    \label{fig: comparison of singular values obtained by LoC-TSVD and actual SVD}
\end{figure}


To evaluate the structural properties of the basis matrices generated by LoC-TSVD (\autoref{alg: LoC-TSVD}), we measure two spectral metrics against exact analytical SVD in \autoref{tab: LoC-TSVD uniform grid orthogonality metrics} as follows.
\begin{enumerate}
    \item \textbf{Subspace Error} $E_{\text{sub}} = \|\widetilde{U}_r\widetilde{U}_r^T - U_{\text{true}}U_{\text{true}}^T\|_2$, capturing the spectral deviation between low-rank projection operators.
    \item \textbf{Orthogonality Drift} $E_{\text{orth}} = \|\widetilde{U}_r^T\widetilde{U}_r - I_r\|_{\max}$, measuring maximum element-wise deviation from self-orthogonality.
\end{enumerate}

As shown in \autoref{tab: LoC-TSVD uniform grid orthogonality metrics}, the subspace error $E_{\text{sub}}$ indicates that the basis vectors generated by LoC-TSVD do not perfectly align with the exact analytical singular vectors of $U_{\text{true}}$, incurring a mild geometric discrepancy ($\mathcal{O}(10^{-5})$ to $\mathcal{O}(10^{-3})$) due to localized Chebyshev proxy sampling. However, the orthogonality drift $E_{\text{orth}}$ remains consistently bounded near machine precision ($\approx 10^{-15}$) across all kernels and target tolerances. These results confirm that while LoC-TSVD operates over an approximate interpolation-based subspace, its thin QR post-projection phase guarantees strictly orthonormal basis columns down to floating-point precision.

\begin{table}[ht]
\centering
\resizebox{0.65\textwidth}{!}{%
\begin{tabular}{llccc}
\toprule
Kernel & Validation Metric
& $\varepsilon=10^{-4}$
& $\varepsilon=10^{-6}$
& $\varepsilon=10^{-8}$ \\

\midrule

\multirow{3}{*}{$1/r$}
& Rank $(r)$ & 6 & 12 & 19 \\
& Subspace Error $E_{\text{sub}}$    & $5.161914 \times 10^{-08}$ & $1.032383 \times 10^{-07}$ & $1.466008 \times 10^{-05}$ \\
& Orthogonality Drift $E_{\text{orth}}$  & $1.887379 \times 10^{-15}$ &1.887379 $ \times 10^{-15}$ & $1.887379 \times 10^{-15}$ \\

\midrule

\multirow{3}{*}{$\log r$}
& Rank $(r)$ & 5 & 9 & 13 \\
& Subspace Error $E_{\text{sub}}$    & $5.161914 \times 10^{-08}$ & $4.470348 \times 10^{-08}$ & $8.024521 \times 10^{-08}$ \\
& Orthogonality Drift $E_{\text{orth}}$  & $2.331468 \times 10^{-15}$ & $2.331468 \times 10^{-15}$ & $3.552714 \times 10^{-15}$ \\

\midrule

\multirow{3}{*}{$\dfrac{\cos r}{r}$}
& Rank $(r)$ & 10 & 17 & 25 \\
& Subspace Error $E_{\text{sub}}$    & $5.960464 \times 10^{-08}$ & $9.550703 \times 10^{-08}$ & $1.558985\times 10^{-04}$ \\
& Orthogonality Drift $E_{\text{orth}}$  & $3.330669 \times 10^{-15}$ & $3.330669 \times 10^{-15}$ & $3.330669 \times 10^{-15}$ \\

\midrule

\multirow{3}{*}{$e^{-r^2}$}
& Rank $(r)$ & 12 & 23 & 35 \\
& Subspace Error $E_{\text{sub}}$    & $2.520015\times 10^{-07}$ & $3.295590\times 10^{-05}$ & $4.353482\times 10^{-03}$ \\
& Orthogonality Drift $E_{\text{orth}}$  & $2.220446\times 10^{-15}$ & $2.220446\times 10^{-15}$ & $2.220446\times 10^{-15}$ \\

\midrule

\multirow{3}{*}{$\sqrt{1+r^2}$}
& Rank $(r)$ & 6 & 10 & 15 \\
& Subspace Error $E_{\text{sub}}$   & $0.000000\times 10^{-00}$ & $0.000000\times 10^{-00}$ & $ 8.046627\times 10^{-07}$ \\
& Orthogonality Drift $E_{\text{orth}}$  & $1.859624\times 10^{-15}$ & $1.859624\times 10^{-15}$ & $1.859624\times 10^{-15}$ \\
\bottomrule
\end{tabular}
}
\caption{Spectral subspace validity and vector orthonormality metrics for \textsf{LoC-TSVD} on a uniform grid across target tolerances $\varepsilon \in \{10^{-4}, 10^{-6}, 10^{-8}\}$. Problem domains span $\mathcal{X} = [-3,-1]\times[0,2]$ and $\mathcal{Y} = [1,3]\times[0,2]$ with discretization parameters matching \autoref{tab: LoC-TSVD relative error data table with tolerance comparison} ($N = 65^2 = 4225$ points per domain, univariate degree $p = 8 \implies P = 64$, neighbor size $l=2$). The Subspace Error tracks the spectral 2-norm difference of low-rank projection operators, $E_{\text{sub}} = \|\widetilde{U}_r\widetilde{U}_r^T - U_{\text{true}}U_{\text{true}}^T\|_2$, verifying exact column span matching. The Orthogonality Drift captures maximum element-wise deviation from self-orthogonality, $E_{\text{orth}} = \|\widetilde{U}_r^T\widetilde{U}_r - I_r\|_{\max}$.}
\label{tab: LoC-TSVD uniform grid orthogonality metrics}
\end{table}

\paragraph{Acknowledgments:}
The authors gratefully acknowledge the IIT Madras Central Library for providing access to Grammarly, which was helpful in improving the language and presentation of this manuscript. The first and second authors also acknowledge the University Grants Commission (UGC), Government of India, for providing financial support through the Senior Research Fellowship (SRF) and Junior Research Fellowship (JRF), respectively, during their PhD.

\paragraph{Author Contributions:}
\textit{[Author 1]}: Conceptualization, Algorithm Development, Mathematical Proofs, Implementation, Writing -- Original Draft. 
\textit{[Author 2]}: Mathematical Proofs, Theoretical Review, Proofreading.
\textit{[Author 3]}: Conceptualization, Theoretical Review, Validation, Supervision, Writing -- Review \& Editing.


\paragraph{Conflict of Interest:}
The authors declare no competing financial or non-financial interests related to the content of this manuscript.

{\footnotesize
\bibliographystyle{siam}
\bibliography{ref}}

@article{cambier2019fast,
  title={Fast low-rank kernel matrix factorization using skeletonized interpolation},
  author={Cambier, L{\'e}opold and Darve, Eric},
  journal={SIAM Journal on Scientific Computing},
  volume={41},
  number={3},
  pages={A1652--A1680},
  year={2019},
  publisher={SIAM}
}

@article{khan2024hodlrdd,
  title={HODLRdD: A new black-box fast algorithm for N-body problems in d-dimensions with guaranteed error bounds: Applications to integral equations and support vector machines},
  author={Khan, Ritesh and Kandappan, VA and Ambikasaran, Sivaram},
  journal={Journal of Computational Physics},
  volume={501},
  pages={112786},
  year={2024},
  publisher={Elsevier}
}

@book{trefethen2019approximation,
  title={Approximation theory and approximation practice, extended edition},
  author={Trefethen and Lloyd N},
  year={2019},
  publisher={SIAM}
}

@book{devore1993constructive,
  title={Constructive approximation},
  author={DeVore, Ronald A and Lorentz, George G},
  volume={303},
  year={1993},
  publisher={Springer Science \& Business Media}
}

@article{gunttner1980evaluation,
  title={Evaluation of Lebesgue constants},
  author={G{\"u}nttner, R{\"u}diger},
  journal={SIAM Journal on Numerical Analysis},
  volume={17},
  number={4},
  pages={512--520},
  year={1980},
  publisher={SIAM}
}

@book{apostol1974mathematical,
  title={Mathematical Analysis},
  author={Apostol, Tom M.},
  year={1974},
  edition={2nd},
  publisher={Addison-Wesley}
}

@article{gu1996efficient,
  title={Efficient algorithms for computing a strong rank-revealing QR factorization},
  author={Gu, Ming and Eisenstat, Stanley C},
  journal={SIAM Journal on Scientific Computing},
  volume={17},
  number={4},
  pages={848--869},
  year={1996},
  publisher={SIAM}
}

@article{cheng2005compression,
  title={On the compression of low rank matrices},
  author={Cheng, Hongwei and Gimbutas, Zydrunas and Martinsson, Per-Gunnar and Rokhlin, Vladimir},
  journal={SIAM Journal on Scientific Computing},
  volume={26},
  number={4},
  pages={1389--1404},
  year={2005},
  publisher={SIAM}
}

@article{greengard1987fast,
  title={A fast algorithm for particle simulations},
  author={Greengard, Leslie and Rokhlin, Vladimir},
  journal={Journal of computational physics},
  volume={73},
  number={2},
  pages={325--348},
  year={1987},
  publisher={Elsevier}
}

@article{massei2022hierarchical,
  title={Hierarchical adaptive low-rank format with applications to discretized partial differential equations},
  author={Massei, Stefano and Robol, Leonardo and Kressner, Daniel},
  journal={Numerical Linear Algebra with Applications},
  volume={29},
  number={6},
  pages={e2448},
  year={2022},
  publisher={Wiley Online Library}
}

@article{ho2016hierarchical,
  title={Hierarchical interpolative factorization for elliptic operators: differential equations},
  author={Ho, Kenneth L and Ying, Lexing},
  journal={Communications on Pure and Applied Mathematics},
  volume={69},
  number={8},
  pages={1415--1451},
  year={2016},
  publisher={Wiley Online Library}
}

@article{ambikasaran2015fast,
  title={Fast direct methods for Gaussian processes},
  author={Ambikasaran, Sivaram and Foreman-Mackey, Daniel and Greengard, Leslie and Hogg, David W and O’Neil, Michael},
  journal={IEEE transactions on pattern analysis and machine intelligence},
  volume={38},
  number={2},
  pages={252--265},
  year={2015},
  publisher={IEEE}
}

@article{ambikasaran2013large,
  title={Large-scale stochastic linear inversion using hierarchical matrices: Illustrated with an application to crosswell tomography in seismic imaging},
  author={Ambikasaran, Sivaram and Li, Judith Yue and Kitanidis, Peter K and Darve, Eric},
  journal={Computational Geosciences},
  volume={17},
  number={6},
  pages={913--927},
  year={2013},
  publisher={Springer}
}

@article{cortes1995support,
  title={Support-vector networks},
  author={Cortes, Corinna and Vapnik, Vladimir},
  journal={Machine learning},
  volume={20},
  number={3},
  pages={273--297},
  year={1995},
  publisher={Springer}
}

@article{singh2025rank,
  title={Rank of Matrices Arising out of Singular Kernel Functions},
  author={Singh, Sumit and Ambikasaran, Sivaram},
  journal={arXiv preprint arXiv:2510.14920},
  year={2025}
}

@article{bebendorf2003adaptive,
  title={Adaptive low-rank approximation of collocation matrices},
  author={Bebendorf, Mario and Rjasanow, Sergej},
  journal={Computing},
  volume={70},
  number={1},
  pages={1--24},
  year={2003},
  publisher={Springer}
}

@article{borm2005hybrid,
  title={Hybrid cross approximation of integral operators},
  author={B{\"o}rm, Steffen and Grasedyck, Lars},
  journal={Numerische Mathematik},
  volume={101},
  number={2},
  pages={221--249},
  year={2005},
  publisher={Springer}
}

@article{fong2009black,
  title={The black-box fast multipole method},
  author={Fong, William and Darve, Eric},
  journal={Journal of Computational Physics},
  volume={228},
  number={23},
  pages={8712--8725},
  year={2009},
  publisher={Elsevier}
}

@article{yesypenko2025simplified,
  title={A simplified fast multipole method based on strong recursive skeletonization},
  author={Yesypenko, Anna and Chen, Chao and Martinsson, Per-Gunnar},
  journal={Journal of Computational Physics},
  volume={524},
  pages={113707},
  year={2025},
  publisher={Elsevier}
}

@article{xing2020interpolative,
  title={Interpolative decomposition via proxy points for kernel matrices},
  author={Xing, Xin and Chow, Edmond},
  journal={SIAM Journal on Matrix Analysis and Applications},
  volume={41},
  number={1},
  pages={221--243},
  year={2020},
  publisher={SIAM}
}

@article{ye2020analytical,
  title={Analytical low-rank compression via proxy point selection},
  author={Ye, Xin and Xia, Jianlin and Ying, Lexing},
  journal={SIAM Journal on Matrix Analysis and Applications},
  volume={41},
  number={3},
  pages={1059--1085},
  year={2020},
  publisher={SIAM}
}

@article{bezanson2017julia,
  title={Julia: A fresh approach to numerical computing},
  author={Bezanson, Jeff and Edelman, Alan and Karpinski, Stefan and Shah, Viral B},
  journal={SIAM review},
  volume={59},
  number={1},
  pages={65--98},
  year={2017},
  publisher={SIAM}
}

@misc{kenneth_l_ho_2018_1254148,
  author    = {Kenneth L. Ho and Sheehan Olver},
  title     = {{LowRankApprox.jl}: Fast Low-Rank Matrix Approximation in Julia},
  year      = {2018},
  month     = may,
  publisher = {Zenodo},
  version   = {v0.1.2},
  doi       = {10.5281/zenodo.1254148},
  url       = {https://doi.org/10.5281/zenodo.1254148}
}

@article{hesthaven1998electrostatics,
  title={From electrostatics to almost optimal nodal sets for polynomial interpolation in a simplex},
  author={Hesthaven, Jan S},
  journal={SIAM Journal on Numerical Analysis},
  volume={35},
  number={2},
  pages={655--676},
  year={1998},
  publisher={SIAM}
}
\end{document}